\documentclass[12 pt]{amsart}

\usepackage{color}
\usepackage{amssymb, amsmath, amsthm}

\usepackage{fullpage}
\usepackage{graphicx}
\usepackage{pinlabel}
\usepackage[hidelinks]{hyperref}
\usepackage{amsrefs}

\newtheorem{theorem}{Theorem}[section]
\newtheorem*{theorem*}{Theorem}
\newtheorem*{theoremBrunnian-tunnel}{Theorem \ref{Brunnian result} (Tunnel Number Version)}
\newtheorem*{theoremBrunnian-bridge}{Theorem \ref{Brunnian result} (Bridge Number Version)}

\newtheorem*{theoremtunnel1}{Theorem \ref{tunnel result 1}}
\newtheorem*{theoremMorimoto}{Theorem \ref{thm: Morimoto}}
\newtheorem*{corollary-thetafact}{Corollary \ref{thetafact}}
\newtheorem*{corollarydistribution}{Corollary \ref{cor: distribution} (Simplified)}
\newtheorem*{theorem-bridgeresult1}{Theorem \ref{bridge result 1}}
\newtheorem{corollary}[theorem]{Corollary}

\newtheorem{lemma}[theorem]{Lemma}
\newtheorem{proposition}[theorem]{Proposition}

\theoremstyle{definition}

\newtheorem{remark}[theorem]{Remark}

\newcommand{\Z}{\mathbb{Z}}

\newcommand{\nil}{\varnothing}

\newcommand{\wihat}[1]{\widehat{#1}}

\newcommand{\defn}[1]{\textbf{#1}}
\renewcommand{\H}{\mathbb{H}}

\newcommand{\punct}[1]{\mathring{#1}}
\newcommand{\g}{\frak{g}}
\renewcommand{\t}{\frak{t}}
\renewcommand{\b}{\frak{b}}

\newcommand{\netchi}{\operatorname{net}\chi} 

\newcommand{\netextent}{\operatorname{netext}}

\newcommand{\extent}{\operatorname{ext}}

\newcommand{\boundary}{\partial}

\newcommand{\mc}[1]{\mathcal{#1}}

\newcommand{\cpt}{\sqsubset}
\newcommand{\thinsto}{\to}

\newcommand{\spacing}{
\parskip 6.6pt
\parindent 0pt
}

\spacing

\begin{document}

\title{Tunnel number and bridge number of composite genus 2 spatial graphs}
   \author{Scott A. Taylor and Maggy Tomova}
   
   \begin{abstract} Connected sum and trivalent vertex sum are natural operations on genus 2 spatial graphs and, as with knots, tunnel number behaves in interesting ways under these operations. We prove sharp lower bounds on the degeneration of tunnel number under these operations. In particular, when the graphs are Brunnian $\theta$-curves, we show that the tunnel number is bounded below by the number of prime factors and when the factors are m-small, then tunnel number is bounded below by the sum of the tunnel numbers of the factors. This extends theorems of Scharlemann-Schultens and Morimoto to genus 2 graphs. We are able to prove similar results for the bridge number of such graphs. The main tool is a family of recently defined invariants for knots, links, and spatial graphs that detect the unknot and are additive under connected sum and vertex sum. In this paper, we also show that they detect trivial $\theta$-curves. 
 \end{abstract}

\dedicatory{
For Martin Scharlemann, in gratitude for years of encouragement and beautiful mathematics.
}
  
 \maketitle  

\section{Introduction}
\subsection{Tunnel number of composite graphs}
If $K$ is a knot, link, or spatial graph properly embedded in a closed 3-manifold $M$, we may embed arcs $\tau_1, \hdots, \tau_n$ (for some $n \geq 0$) in $M$ so that they are pairwise disjoint, have endpoints on $K$, interiors disjoint from $K$, and so that the exterior of the spatial graph $K \cup \tau_1 \cup \cdots \cup \tau_n$ in $M$ is a handlebody. The minimum number $\t(K)$ of arcs needed is the \defn{tunnel number} of $K$. The behavior of tunnel number for knots under connected sum of knots is rather mysterious. It is well known (and easy to prove) that for all knots $K_1$ and $K_2$, $\t(K_1 \# K_2) \leq \t(K_1) + \t(K_2) + 1$. There are examples of knots $K_1$ and $K_2$ in $S^3$, such that the inequality is sharp (see \cites{MR, MSY}) and other examples where the inequality is strict. In fact, the difference $\t(K_1) + \t(K_2) - \t(K_1 \# K_2)$ can be quite large \cite{Kobayashi}.  Scharlemann and Schultens \cite{SS} proved the well-known result that the sum of $n$ prime knots has tunnel number at least $n$. Morimoto \cite{Morimoto15} characterized the nontrivial knots $K_1$ and $K_2$ such that $\t(K_1 \# K_2) = 2$.  In particular, at least one of them must be a  2-bridge knot or (1,1) knot (see below, for the definition).  He also showed that the connected sum of $m$-small knots in 3-manifolds without lens space summands will have tunnel number at least the tunnel number of the summands.  See \cite{Schirmer} for a good overview of what is known concerning tunnel number for knots.  

As with knots, we can form the connected sum of trivalent spatial graphs and ask about the behavior of tunnel number. Eudave-Mu\~noz and Ozawa studied composite tunnel number 1 genus 2 spatial graphs where one summand is a knot \cite{EMO}. However, we can also perform vertex sums and ask how tunnel number behaves. For $\theta$-curves, the vertex sum acts as a connected sum on cycles. Additionally, passing to branched double covers over a cycle lifts the vertex sum of $\theta$-curves to the connected sum of knots, so we might expect tunnel number of composite $\theta$-curves to behave similarly to knots. However, things are not that simple. Deferring some definitions until later, our main result is:

\begin{theoremtunnel1}
Suppose that $(M,T)$ is an irreducible composite (3-manifold, graph) pair such that every sphere in $M$ separates and $T$ is a genus 2 graph. Then
\[
\t(M,T) \geq \frac{m - 1}{2} + k
\]
where $m$ is the number of factors in a prime factorization that are genus 2 graphs which are not the trivial $\theta$-curves or Hopf graphs and $k$ is the number of factors that are knots which are not $(1,0)$-curves.
\end{theoremtunnel1}

A spatial $\theta$-curve in $S^3$ never has a Hopf graph (which is a kind of handcuff graph) or a (1,0)-curve (which is a core loop in a lens space) as a factor. Furthermore, if a $\theta$-curve has a trivial $\theta$-curve as a factor in a prime factorization then it was obtained by tying nontrivial local knots in some of the edges of a trivial $\theta$-curve. (See below for precise definitions.) Thus, we immediately have the corollary:

\begin{corollary}\label{thetacor}
If $T$ is a composite spatial $\theta$-curve in $S^3$ of tunnel number 1, then it is either the vertex sum of two or three prime $\theta$-curves, or is the connected sum of a prime $\theta$-curve and a nontrivial knot, or is the result of tying one nontrivial knot in an edge of a trivial $\theta$-curve.
\end{corollary}

The lower bound in Theorem \ref{tunnel result 1} is sharp. Figure \ref{thetatunnel1} shows an example of a tunnel number one $\theta$-curve \[(S^3, \theta) = (S^3, \theta_1) \#_3 (S^3, \theta_2) \#_3 (S^3, \theta_3)\] with each pair nontrivial. The example is readily adapted to provide an example of the vertex sum of two nontrivial $\theta$-curves that has tunnel number one and an example of the vertex sum of $2n + 1$ $\theta$-curves having tunnel number $n$.  In those examples, each factor has tunnel number 0. For another example, the vertex sum of the Kinoshita graph \cite{Kinoshita} with any 2-bridge $\theta$-graph (see below for the definition) also has tunnel number 1. The Kinoshita graph has tunnel number 1 and the 2-bridge $\theta$-graphs have tunnel number 0.

\begin{figure}[ht!]
\includegraphics[scale=0.3]{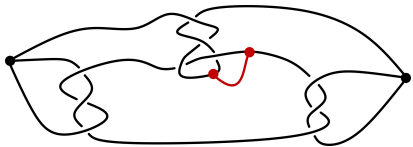}
\caption{A $\theta$-curve in $S^3$ that is the vertex sum of three nontrivial $\theta$-curves and which has tunnel number one. The $\theta$-curve is in black and an unknotting tunnel is drawn in red.}
\label{thetatunnel1}
\end{figure}

As previously discovered by Eudave-Mu\~noz and Ozawa there are also examples involving connected sum. Consider a tunnel number one knot $K$ in $S^3$. Let $\tau$ be a tunnel for $K$ with distinct endpoints on $K$. Tie a 2-bridge knot in $\tau$, to obtain the arc $\tau'$ and set $T = K \cup \tau'$. See Figure \ref{thetatunnel2} for an example.  Notice that $T$ is the connect sum of a nontrivial $\theta$-curve with a knot. An unknotted arc that is a tunnel for $\tau'$, is then also a tunnel for all of $T$. The paper \cite{EMO} gives a number of other examples of tunnel number 1 $\theta$-curves and handcuff curves that have a knot summand.

\begin{figure}[ht!]
\includegraphics[scale=0.5, angle=90]{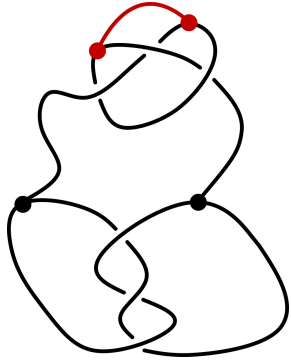}
\caption{A tunnel number one $\theta$-curve in $S^3$ that is the vertex sum of a nontrivial $\theta$-curve and a nontrivial knot. The $\theta$-curve is in black and an unknotting tunnel is drawn in red.}
\label{thetatunnel2}
\end{figure}

On the other hand, we do get the Scharlemann-Schultens lower bound if we consider only the class of \emph{Brunnian} $\theta$-curves. A spatial graph in $S^3$ is \defn{Brunnian} if every proper subgraph can be isotoped into a (tame) sphere, but the graph itself cannot be.  In particular, a nontrivial $\theta$-curve in $S^3$ is Brunnian if and only if every cycle is unknotted. It is easily shown that the factors in a prime factorization of a Brunnian $\theta$-curve or handcuff curve are also Brunnian\footnote{The terms \defn{almost unknotted}, \defn{minimally knotted}, and \defn{ravel} are also used in the literature; for graphs with more than 3 edges these terms may all have slightly different meanings, depending on the authors. }. Here is the first part of the statement of Theorem \ref{Brunnian result}.

\begin{theoremBrunnian-tunnel}
Suppose that $T \subset S^3$ is a composite Brunnian $\theta$-curve with $m$ factors in a prime factorization. Then
\[
\t(S^3,T) \geq m.
\]
\end{theoremBrunnian-tunnel}

The Kinoshita graph \cite{Kinoshita} is easily seen to have tunnel number one; so we observe both that the Kinoshita graph is prime and that the tunnel number of the trivalent vertex sum of $m$ copies of Kinoshita graph is at least $m$. Primality of the Kinoshita graph was previously known; see \cite[Example 2.5]{Ozawa} or \cite[Example 3.1]{CalcutMetcalfBurton}, for example. Makoto Ozawa pointed out to us that by \cites{GR, Morimoto96}, nontrivial $\theta$-curves of tunnel number 0 in $S^3$ are prime. The most significant prior work on the tunnel number of composite genus 2 graphs was done by Eudave-Mu\~noz and Ozawa \cite{EMO}. They classified all tunnel number one $\theta$-curves and handcuff curves that are the connected sum of a genus 2 curve and a knot, but did not consider trivalent vertex sum. Our methods could likely recapture and generalize both Morimoto's classification and Eudave-Mu\~noz and Ozawa's results. We have, however, avoided doing that in the interests of space. 

As with any inequality, we may ask under what circumstances (if any) the inequality is sharp. As we mentioned, Morimoto \cite{Morimoto15} studied this question when he analyzed the factors of a composite tunnel number two knot. We analyze when equality in Theorem \ref{tunnel result 1} holds. As in previous work, we find that an important role is played by so-called 2-bridge knots and graphs, Hopf graphs, and tunnel number one knots and graphs. Deferring definitions until later, (and stating the result only for $\theta$-curves) we show:

\begin{corollary-thetafact}
Suppose that $(M,T)$ is a connected, irreducible, composite pair with $T$ a $\theta$-graph and every sphere in $M$ separates. Also assume that no factor in a prime factorization of $(M,T)$ is a knot or $(0,2)$-curve\footnote{A $(0,2)$-curve is the genus 2 graph version of a 2-bridge knot and a (1,1)-curve is the genus 2 graph version of a knot in $S^3$ that is 1-bridge with respect to a Heegaard torus.}. If $(M,T)$ has $m$ factors and $\t(M,T)= \frac{m-1}{2}$, then $T$ has exactly 3 factors and they are all $(1,1)$-curves.
\end{corollary-thetafact}

More generally, we show

\begin{corollarydistribution}
Suppose that $(M,T)$ is a composite, connected, irreducible pair such that every sphere in $M$ separates and $T$ is a genus 2 curve.  Suppose that $(M,T)$ has $n$ factors, of which $m$ are genus 2 graphs that are not the trivial $\theta$-curve or a Hopf graph and $k$ of which are knots that are not $(1,0)$-curves.  If
\[
\t(M,T) = \frac{m-1}{2} + k
\]
then the number of factors that are trivial $\theta$-curves, trivial 2-bouquets, $(0,2)$-curves, or (1,1)-knots is at least $(n-3)/3$. 
\end{corollarydistribution}

The more general version of the theorem elaborates on the proportion of other types of spatial graph types showing up as the factors in a prime factorization where equality in Theorem \ref{tunnel result 1} is achieved.

We also prove a version of Morimoto's theorem for m-small knots\footnote{We note that even when all the factors are m-small, this result is not necessarily stronger than Corollary \ref{thetacor} since a spatial graph may have tunnel number 0 (i.e. have handlebody complement) but still contribute a positive amount to the tunnel number of a composite graph of which it is a factor. This is similar to how, for knots, the connected sum of two knots that are the cores of lens spaces (and thus have tunnel number 0) must have tunnel number one since the ambient 3-manifold has Heegaard genus 2.}.

\begin{theoremMorimoto}
Suppose that $(M,T)$ is an irreducible, composite pair with $T$ a $\theta$-curve or handcuff curve and where every sphere in $M$ separates. Let $(\wihat{M}_1, \wihat{T}_1), \cdots, (\wihat{M}_n, \wihat{T}_n)$ be the factors of a prime factorization of $(M,T)$ and suppose that each is m-small. Then
\[
\t(M,T) \geq \t(\wihat{M}_1, \wihat{T}_1) + \cdots + \t(\wihat{M}_n, \wihat{T}_n).
\]
\end{theoremMorimoto}

\subsection{Bridge number of composite graphs}

For a knot, link, or spatial graph $T \subset S^3$, a \defn{bridge sphere} for $T$ is a sphere $H \subset T$ such that in the 3-balls on either side of $H$ there is a properly embedded disc containing the portions of $T$ on that side of the sphere and $T \setminus H$ is acyclic.\footnote{There has been some disagreement over the proper definition of bridge number, see \cites{Goda, Motohashi1}. Our definition is the same as in, for example, \cites{Motohashi1, Ozawa-bridge}.} The \defn{bridge number} $\b(T)$ is the minimum of $|H \cap T|/2$ over all bridge spheres for $T$. When $T$ is a knot or link, the bridge number is a positive integer. When $T$ is a spatial graph, the bridge number is a positive integer or half integer. The unknot is the unique knot of bridge number 1 and the trivial $\theta$-curve is the unique $\theta$-curve of bridge number 3/2.  Schubert's well known theorem \cite{Schubert} says that the quantity $\b - 1$ is additive under connected sum of knots. In particular, knots of bridge number 2 are prime. Inspired by that result, we might hope that $\theta$-curves of bridge number 2 are also prime. Motohashi \cite{Motohashi1} shows that this is not the case. In particular, the trivalent vertex sum of two 2-bridge $\theta$-curves can also have bridge number 2. She also shows that any composite bridge number 2 $\theta$-curve has factors that are 2-bridge and that such $\theta$-curves are not Brunnian. In fact, they are the union of an arc (actually a tunnel) with a 2-bridge knot\footnote{This is generalized in \cite{Ozawa}, where a classification of tangle decompositions of 2-bridge $\theta$-curves and handcuff curves is given.}. We improve on this and show:
\begin{theorem-bridgeresult1}
Suppose that $T \subset S^3$ is an irreducible composite genus 2 graph. Then
\[
\b(T) \geq \frac{m+3}{2} + k 
\]
where $m$ is the number of factors that are genus 2 graphs which are not the trivial $\theta$-curve and $k$ is the number of factors which are knots. Furthermore, if equality holds then every factor in a prime factorization of $T$ is a $(0,2)$-curve, trivial $\theta$-curve, or trivial 2-bouquet.
\end{theorem-bridgeresult1}

As with tunnel number, we get a stronger result for Brunnian $\theta$-curves. 
\begin{theoremBrunnian-bridge}
Suppose that $T \subset S^3$ is a Brunnian composite $\theta$-curve having $m$ factors in its prime factorization. Then
\[
\b(T) \geq m + \frac{3}{2}
\]
\end{theoremBrunnian-bridge}

The Kinoshita graph is an example of a Brunnian $\theta$-graph of bridge number 5/2, so we can also use the bridge number version of Theorem \ref{Brunnian result} to conclude it is prime. In \cite{JKLMTZ}, the authors construct Brunnian $\theta$-graphs with bridge number at most 3. By Theorem \ref{Brunnian result}, they are prime. 

Doll \cite{Doll} introduced bridge numbers with respect to higher genus surfaces. His definition can be adapted to spatial graphs. Our methods would also provide lower bounds on those invariants with respect to the number of factors. In the interests of space, we do not pursue this.

\subsection{The strategy}

In \cite{TT2}, we introduced new invariants of (3-manifold, graph pairs) $(M,T)$ and proved those invariants were additive under connected sum and $(-1/2)$-additive under trivalent vertex sum. One family of invariants we called ``net extent'' and denoted it by $\netextent_x(M,T)$ where $x$ is any even integer at least $2\g(M) -2$, and $\g(M)$ is the Heegaard genus of $M$.  For each $x$, the invariant $\netextent_x(M,T)$ is a non-negative integer or half-integer. For a fixed pair  $(M,T)$, $\netextent_x(M,T)$ is a decreasing sequence in $x$ and so is eventually constant at some term, which we denote by $\netextent_\infty(M,T)$. In addition to behaving well under sums, it also (in a certain sense) detects the unknot. Furthermore,  if $H \subset M\setminus T$ is a Heegaard surface for the exterior of $T$ in $M$, then
\[
\g(H) - 1 \geq \netextent_\infty(M,T)
\]
and if $T$ is a spatial graph in $S^3$, then also
\[
\b(T) - 1 \geq \netextent_\infty(M,T)
\]
Thus, we can use the additivity properties of net extent to derive lower bounds on the tunnel number and bridge number of spatial graphs.  In \cite[Theorem 7.6]{TT2}, we applied this philosophy to prove generalizations of the Scharlemann-Schultens theorem and Morimoto's theorems for knots. The purpose of this paper is to apply the same philosophy to \defn{genus 2 spatial graphs}; that is, connected graphs of Euler characteristic -1, embedded in a 3-manifold. To that end, it is helpful to briefly review the strategy.

Suppose that $(M,T)$ is composite and satisfies certain other mild hypotheses we will explain later. Let $H$ be either a bridge sphere for $T$ or a Heegaard surface for $M \setminus T$. Using the definition and additivity properties of net extent, we are able to conclude that we have
\[
\frac{-\chi(H) + |H \cap T|}{2} \geq \netextent_\infty(M,T) = c + \sum\limits_{i=1}^n \netextent_\infty(\wihat{M}_i, \wihat{T}_i)
\]
where $(\wihat{M}_i, \wihat{T}_i)$ for $i = 1, \hdots, n$ are the factors in a particular prime decomposition of $(M,T)$ and $c$ is a constant depending (in a very weak way) on the decomposition. Our most basic lower bounds on the tunnel number and bridge number of a composite graph are obtained by bounding $\netextent_\infty(\wihat{M}_i, \wihat{T}_i)$ below for each $i$. When $\wihat{T}_i$ is a knot, the unknot detection properties proved in \cite{TT2} are what we need. When $\wihat{T}_i$ is a genus 2 graph, we prove that in most cases, $\netextent_\infty(\wihat{M}_i, \wihat{T}_i) \geq 1$. In Section \ref{sec:pairs of low extent}, we define the graph types which turn out to represent all genus 2 spatial graphs having net extent 1 and show they are not Brunnian. In Section \ref{sec:small extent},  we prove the classification of the genus 2 graphs having $\netextent_\infty(\wihat{M}_i, \wihat{T}_i) = 1$. With some exceptions, these correspond to spatial graph-theoretic versions of tunnel number 1 knots. This allows us to draw conclusions about the factors of a composite genus 2 spatial graph achieving the minimum tunnel number or bridge number relative to the number of components. It also lets us prove our lower bound on the tunnel number and bridge number of composite Brunnian graphs. 

Section \ref{sec: notationterm} introduces notation and terminology, including the definition of net extent.  In Section \ref{sec:thin position} we introduce the notion of thin position which is key to proving our results.  In Section \ref{sec:vp-compressionbodies} we analyze vp-compressionbodies of low complexity. In Sections \ref{sec:pairs of low extent} and \ref{sec:small extent} we discuss types of (graph, manifold)-pairs of low net extent. In Section \ref{sec:lower bound} we prove the lower bound results and finally in Section \ref{sec:equality} we study the cases where equality is achieved.

\subsection{Acknowledgements} 
Thanks to Makoto Ozawa for helpful comments on the history of the topology of spatial $\theta$-curves. Taylor was supported by a research grant from Colby College and Tomova was supported by a grant from the NSF.

\section{Notation and Terminology}\label{sec: notationterm}

We follow terminology introduced in \cite{TT1} and \cite{TT2}; which in turn was inspired by \cite{HS}. All 3-manifolds and surfaces we encounter are compact and orientable. For submanifolds $X,Y$ of a 3-manifold $M$, we let $X \setminus Y$ denote the complement of an open regular neighborhood of $Y$ in $X$ and $|X|$ the number of connected components of $X$. So, for example, if $K \subset M$ is a knot, then $M \setminus K$ is the exterior of $K$. We write $X \cpt Y$ to mean that $X$ is a path-component of $Y$. The \defn{genus} $\g(M)$ of a 3-manifold $M$ is the minimum $g$ such that $M$ admits a Heegaard surface of genus $g$. For any connected spatial graph $T$ in a closed 3-manifold $M$, the tunnel number $\t(M,T) = \g(M\setminus T) + \chi(T) - 1$, where $\g(M\setminus T)$ is the Heegaard genus of the exterior of $T$ in $M$ and $\chi(T)$ is the Euler characteristic of $T$.

\subsection{Pairs and Prime Factorizations}

A (3-manifold, graph) \defn{pair} $(M,T)$ consists of a compact, orientable 3-manifold (possibly with boundary) and a properly embedded graph (i.e. 1--complex) $T \subset M$ such that no vertex of $T$ has degree 2 and no component of $\boundary M$ is a sphere intersecting $T$ two or fewer times. Usually we also assume that every sphere in $M$ separates $M$, although this assumption could be weakened. Its use arises from some facts we appeal to from \cite{TT2} and in Theorem \ref{unique roots} below.  We do allow $T$ to have components that are closed loops with no vertices.  As $T$ is embedded in a 3-manifold, we say it is a \defn{spatial graph}. A spatial genus 2 graph with a single vertex is a \defn{2-bouquet}; a spatial genus 2 graph with no loops is a \defn{$\theta$-curve}; a spatial genus 2 graph with two loops and one separating edge is a \defn{handcuff curve}. Figure \ref{AbstractTypes} depicts the abstract graph type of each type of genus 2 graph. These spatial graphs (and their regular neighborhoods, spatial genus 2 handlebodies) have recently gained attention for their rich topological, algebraic, and geometric structure and their applications to the study of certain biological processes (e.g. \cite{BuckODonnol}). Additionally, they make appearances in knot theory due to their connections with the study of tunnel number one knots and links (e.g. \cite{CM}), as well as other invariants such as unknotting number (e.g. \cite{Lackenby}).  As much as possible, we work with spatial graphs more generally. Some of the auxiliary results of this paper should prove useful in future work.

\begin{figure}[ht!]
\centering
\includegraphics{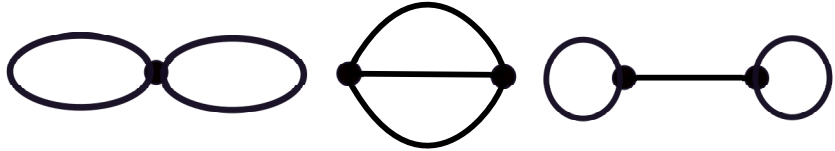}
\caption{The abstract graph types of the 2-bouquet, $\theta$-curve, and handcuff curve}
\label{AbstractTypes}
\end{figure}

If $T_1$ and $T_2$ are spatial graphs in 3-manifolds $M_1$ and $M_2$, we can form the connected sum $M_1 \# M_2$ of the ambient 3-manifolds by choosing points $p_1 \in M_1$ and $p_2 \in M_2$, removing a small regular neighborhood of each, and then gluing the new boundary spheres together by some homeomorphism $\psi$. If $p_1$ and $p_2$ are both internal to edges of $T_1$ and $T_2$ we arrive at the \defn{connected sum} $(M_1, T_1) \# (M_2, T_2)$, assuming we choose $\psi$ to take the punctures on one boundary sphere to the punctures on the other boundary sphere. If $p_1$ and $p_2$ are both vertices of degree $k \geq 3$, then we can similarly define the \defn{$k$-valent vertex sum} $(M_1, T_1)\#_k (M_2, T_2)$. The case when $k = 3$ (the \defn{trivalent vertex sum}) is the most important and has been extensively studied for $\theta$-curves in $S^3$. See \cite{Wolcott} for basic results and Figure \ref{AbstractSum} for a schematic depiction of connected sum and trivalent vertex sum for graphs in $S^3$. For $k \geq 4$, these sums are substantially less well-behaved. 

\begin{figure}[ht!]
\labellist
\small\hair 2pt
\endlabellist
\includegraphics[scale=0.4]{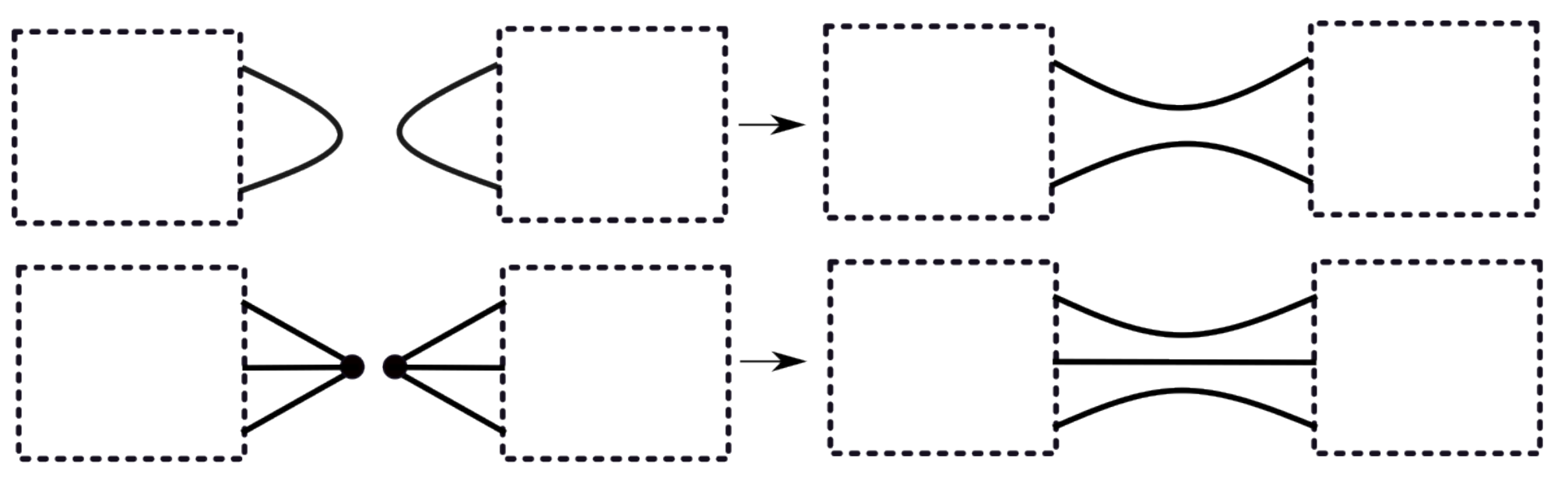}
\caption{The top row depicts the connected sum and the bottom row depicts a trivalent vertex sum.}
\label{AbstractSum}
\end{figure}

The trivalent vertex sum is a particularly natural operation on the set of (3-manifold, graph) pairs $(M,T)$ where $T$ is a $\theta$-curve or handcuff curve (see Figure \ref{AbstractTypes} for a depiction of the abstract graph types). Matveev and Turaev \cite{MT} show that the prime factorization of a pair $(M,T)$ such that every sphere in $M$ is separating and $T$ a $\theta$-curve is unique up to orientation choices and a certain equivalence related to the fact that connected sum for knots is commutative. Motohashi \cite{Motohashi2} has a similar result for handcuff curves in $S^3$. See Theorem \ref{unique roots} below for the version we use.

If $S \subset M$ is a properly embedded surface transverse to $T$, we write $S \subset (M,T)$. The points $S \cap T$ are the \defn{punctures} on $S$. A curve on $S$ is \defn{essential} if it does not bound either an unpunctured disc or a once-punctured disc on $S$. Throughout we use various generalizations of compressing discs. An \defn{sc-disc} for a surface $S \subset (M,T)$ or a component $S \subset \boundary M$ is a disc $D$ with interior disjoint from $S$, transverse to $T$, with $\boundary D \subset S$, and with $|D \cap T| \leq 1$ such that $D\setminus T$ is not properly isotopic, relative to $\boundary D$, in $M\setminus T$ into $S \setminus T$. If $\boundary D$ is essential in $S \setminus T$ and $|D \cap T| = 0$, then $D$ is a \defn{compressing disc}; if $|D \cap T| = 0$, but $\boundary D$ is inessential in $S \setminus T$, then $D$ is a \defn{semicompressing disc}. Analogously, if $|D \cap T| = 1$, then $D$ is a \defn{cut-disc} or \defn{semicut-disc} according to whether or not $\boundary D$ is essential or inessential in $S \setminus T$. If $(M,T)$ is irreducible (the definition is below), then $S$ has no semicompressing discs and if $D$ is a semicut-disc then $\boundary D$ bounds a once-punctured disc in $S$. If $S \subset (M,T)$ is a surface such that there is a compressing or cut-disc for $S$ in $(M,T)$, then $S$ is \defn{c-compressible}; otherwise $S$ is \defn{c-incompressible}. If $S \subset (M,T)$ is a sphere bounding a 3-ball disjoint from $T$, or if $S\setminus T$ is $\boundary$-parallel in $M \setminus T$, or if $S$ is c-compressible (respectively compressible), then $S$ is \defn{c-inessential} (respectively \defn{inessential}); otherwise, $S$ is \defn{c-essential} (respectively \defn{essential}). Notice that a surface $S \subset (M,T)$ may be such that $S \setminus T$ is $\boundary$-parallel in $M \setminus T$, even if $S$ is not $\boundary$-parallel in $M$, as the surface may be partially parallel to portions of $\boundary M$ and partially parallel to portions of the graph.

A pair $(M,T)$ is \defn{connected} if $M$ is connected, though $T$ need not be. It is \defn{trivial} if $M = S^3$ and if $T$ is isotopic into a tame sphere in $M$.  A pair $(M,T)$ is \defn{irreducible} if there is no essential sphere in $(M,T)$ disjoint from $T$ or intersecting $T$ in exactly one point\footnote{Some authors would also require that the exterior of the graph be $\boundary$-irreducible, though we do not. Our terminology is inspired by that for Heegaard splittings.}. A connected, irreducible nontrivial pair $(M,T)$ is \defn{prime} if there is no essential twice or thrice-punctured sphere $S \subset (M,T)$ that separates $M$. A nontrivial, irreducible pair $(M,T)$ is \defn{composite} if there is such a sphere. If the 3-manifold $M$ is clear, we will refer to $T$ as being trivial, or prime, or irreducible, or composite, etc. 

The trivial handcuff curve in $S^3$ is not irreducible, so it will not appear in what follows. One reason to implement the irreducibility hypothesis arises from the fact that if we take the connected sum of a trivial handcuff curve with a nontrivial knot in such a way that the summing point on the handcuff curve lies on the separating edge, the result is isotopic to the trivial handcuff curve. In particular the trivial handcuff curve can be decomposed as a nontrivial connected sum. 

A \defn{prime decomposition} of an irreducible, nontrivial pair $(M,T)$ is the realization of $(M,T)$ as the iterated connected sum and trivalent vertex sum of prime pairs and trivial pairs that are not the unknot.  For each trivial $\theta$-curve in the decomposition, we require that only connected sums (and not trivalent vertex sums) be performed on it. The pairs $(\wihat{M}_i, \wihat{T}_i)$ for $i = 1, \hdots, n$, that are being summed are called the \defn{factors} of the prime factorization. Notice that not all factors in a prime factorization of a connected, irreducible pair $(M,T)$ need be prime. However, if $T$ is a genus 2 curve, then in any prime factorization there is at most one trivial pair and, if there is, it is either a trivial $\theta$-curve or a trivial 2-bouquet and $T$ is the result of tying nontrivial knots in its edges.

\begin{remark}
Our definition of prime factorization of pairs $(M,T)$ with $T$ a $\theta$-curve differs slightly from that in other work, such as \cite{MT}. However, the difference is only significant when it comes to $\theta$ graphs in $S^3$ obtained by tying knots in the edges of a trivial $\theta$-graph. Even in those cases, it is relatively easy to move between the different definitions.
\end{remark}

On a few occasions we will use the following theorem, due to Hog-Angeloni and Matveev \cite{HAM}. We could avoid its use for most of the paper at the expense of making some statements somewhat more complicated. We also note that this theorem is not  a ``unique factorization'' result as commonly understood (see for example, \cites{MT, Motohashi1}) since such a result should also take into account the location of where the sums are performed and, ideally, handle nonseparating spheres.

\begin{theorem}[Hog-Angeloni--Matveev]\label{unique roots}
Suppose that $(M,T)$ is an irreducible pair such that every sphere in $M$ is separating. Then any two prime factorizations of $(M,T)$ contain the same prime factors. Consequently, if $T$ is a genus 2 curve or knot, any two prime factorizations contain the same factors (prime or not).
\end{theorem}
\begin{proof}
This theorem is not stated as such in \cite{HAM}. We briefly explain how to obtain it from their work.  Sections 3 and 7 of \cite{HAM}, while not dealing strictly with (3-manifold, graph) pairs illuminate the distinction between their theory of roots and prime decompositions: the difference lies primarily in how nonseparating spheres are handled. According to \cite[Section 5]{HAM}, pairs are considered up to the equivalence relation generated by pairwise homeomorphism and disjoint union with trivial pairs $(M_0, T_0)$ where $T_0$ is a $\theta$-curve or unknot. The equivalence class of a pair $(M'',T'')$ is obtained from the equivalence class of a pair $(M', T')$ by an \defn{edge move} if it is obtained by decomposing $(M', T')$ along an essential sphere in $(M', T')$ intersecting $T'$ in at most three points. Edge moves induce a partial order on the set of equivalence classes of pairs. A \defn{root} for $(M,T)$ is a minimal element in this partial order that is less than or equal to the class of $(M,T)$. Theorem 7 of \cite{HAM} says that $(M,T)$ has a root $R(M,T)$ and this root is unique. Two choices of representatives for the class $R(M,T)$ are pairwise homeomorphic after discarding the components that are trivial $\theta$-curves and unknots. The components of a representative for $R(M,T)$ that are not trivial $\theta$-curves or unknots are factors of a prime decomposition of $(M,T)$. Conversely, since every sphere in $M$ is separating and since $(M,T)$ is irreducible, a prime factorization (in our sense) of $(M,T)$ results in a (not necessarily connected) pair $(M',T')$ containing no essential spheres with three or fewer punctures. Performing decompositions along  spheres sequentially shows that the prime factors of $(M',T')$ are a representative of the class of the root of $(M,T)$. The result follows from our definition of ``prime'' and ``prime factorization.''
\end{proof}

For a graph $\Gamma$, the \defn{leaves} of $\Gamma$ are the vertices of degree 1; the \defn{internal vertices} of $T$ are those of higher degree.  Suppose that $(M,T)$ is a pair. We let $(\punct{M}, \punct{T})$ denote the pair that results from removing an open regular neighborhood of the internal vertices of $T$ from both $M$ and $T$.  

\subsection{Surfaces and vp-compressionbodies}

For a (3-manifold, graph) pair $(M,T)$ and an embedded surface $S \subset (M,T)$ we write $(M,T)\setminus S$ to mean $(M\setminus S, T \setminus S)$.  A \defn{component} of $(M,T)\setminus S$ is a pair $(M_0, M_0 \cap T)$ where $M_0$ is a connected component of $M \setminus S$. Likewise, $(M_0, T_0) \cpt (M,T)\setminus S$ means that $M_0$ is a component of $M \setminus S$ and $T_0 = T \cap M_0$. If $F \subset (M,T)$ is a punctured sphere with at least 3 punctures, to \defn{surger} $(M,T)$ along $F$ means that we cut $(M,T)$ open along $F$ to obtain a pair $(M', T')$ with two new spherical boundary components and then crush those new boundary components to vertices. To surger along a twice-punctured sphere, we do the same thing but then absorb the resulting degree two vertices into an edge.

For a surface $S \subset (M,T)$, we define the \defn{extent} of $S$ to be
\[
\extent(S) = \frac{1}{2}\Big(-\chi(S) + |S \cap T|\Big)
\]
where $\chi(S)$ is the Euler characteristic of $S$.

A pair $(C,T_C)$ is a \defn{trivial ball compressionbody} if $C = B^3$ and $T_C \subset C$ is an unknotted arc. The pair $(C,T_C)$ is a \defn{trivial product compressionbody} if it is homeomorphic to $(F \times I, \text{ vertical arcs})$ for a closed, connected surface $F$. A connected pair $(C,T_C)$ with a component of $\boundary C$ specified as $\boundary_+ C$ is a \defn{vp-compressionbody}\footnote{The ``vp'' stands for ``vertex-punctured.''} if there exists a collection of pairwise disjoint sc-discs $\Delta$ for $\boundary_+ C$ in $(C,T_C)$ such that $(C,T_C)\setminus \Delta$ is the disjoint union of trivial ball compressionbodies and trivial product compressionbodies. For a more complete analysis of vp-compressionbodies, see \cite{TT1}. We let $\boundary_- C = \boundary C \setminus \boundary_+ C$; it may be the case that $\boundary_- C = \nil$, however $\boundary_+ C \neq \nil$. See Figure \ref{fig:vpcompbdy} for an example. Observe that $(\punct{C}, \punct{T}_C)$ is also a vp-compressionbody, with $\boundary_+ \punct{C} = \boundary_+ C$ and $\boundary_- \punct{C}$ the union of $\boundary_- C$ with spherical components corresponding to the internal vertices of $T_C$. The components of $\punct{T}_C$ can be partitioned into four types: vertical arcs (arcs with one endpoint on $\boundary_- \punct{C}$); bridge arcs (arcs with both endpoints on $\boundary_+ C$); ghost arcs (arcs with both endpoints on $\boundary_- \punct{C}$); and core loops (components disjoint from $\boundary \punct{C}$). The \defn{ghost arc graph} for $(C, T_C)$ is the graph with vertices the components of $\boundary_- \punct{C}$ and edges the ghost arcs of $\punct{T}_C$. A \defn{spine} for a compressionbody $C$ is the union of $\boundary_- C$ with a properly embedded graph $\Gamma \subset C$ such that $C$ deformation retracts to $\boundary_- C\cup \Gamma$ and $\Gamma$ is disjoint from $\boundary_+ C$. In many of the arguments that follow, it will often be the case that the  ghost arc graph of $(C, T_C)$ is a spine for $C$.

\begin{figure}
\centering
\includegraphics[scale=0.4]{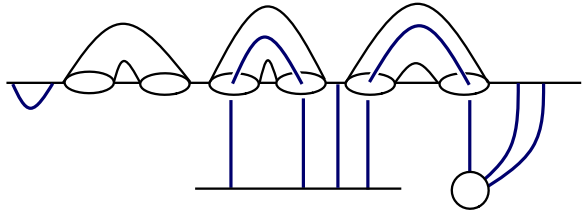}
\caption{A vp-compressionbody $(C,T)$ such that $g(\boundary_+ C) - g(\boundary_- C)  = 3$ and $\boundary_- C$ has two components, one of which is a thrice-punctured sphere. There are two ghost arcs, three vertical arcs, and one bridge arc.}
\label{fig:vpcompbdy}
\end{figure}

We will rely on two technical apparatuses for vp-compressionbodies. For a vp-compressionbody $(C, T_C)$, we define its \defn{index} to be 
\[
\delta(C,T_C) = \delta(\punct{C}, \punct{T}_C) = \extent(\boundary_+ C) - \extent(\boundary_- \punct{C}).
\]
We showed in \cite[Lemma 4.2]{TT2} that $\delta(C, T_C) \geq 0$. Notice that index is an integer since the Euler characteristic of a closed surface is even and since each edge of $\punct{T}_C$ either has both end points on $\boundary_+ C$, both endpoints on $\boundary_- \punct{C}$, or one endpoint on each. 

We use the next lemma throughout the paper, usually without comment.  

\begin{lemma}\label{ghostarcgraphlem}
Suppose that $(C,T)$ is a vp-compressionbody. Let $S$ be the surface that is a frontier of a regular neighborhood of $\boundary_- \punct{C}$ with the ghost arcs of $T$. Then $g(\boundary_+ C) \geq g(S)$. 
\end{lemma}

\begin{proof}[Proof Sketch]
 Take a complete collection of sc-discs $\Delta$ for $\boundary_+ \punct{C}$. Let $T'$ be the union of all the ghost-arcs of $\punct{T}$. By the definition of vp-compressionbody, $(C,T')\setminus \Delta$ is the union of trivial vp-compressionbodies. Reconstructing $\boundary_+ C$ from $\boundary_- C$ by undoing the compressions produces the desired result.
\end{proof}

The following corollary will be very useful and follows immediately.

\begin{corollary}\label{ghostarcgraph cor}
Suppose that $(C,T)$ is a vp-compressionbody and that $\Gamma$ is the ghost arc graph.
\begin{enumerate}
\item If $\boundary_+ C$ is a sphere, then so is each component of $\boundary_- C$ and $\Gamma$ is acyclic.
\item If $\boundary_+ C$ is a torus, then $\boundary_- C$ is the union of spheres and at most one torus. If $\boundary_- C$ contains a torus, then $\Gamma$ is acyclic. If $\boundary_- C$ does not contain a torus, then $\Gamma$ contains at most one cycle.
\end{enumerate}
\end{corollary}

\subsection{Multiple vp-bridge surfaces and net extent}

A connected surface $H \subset (M,T)$ is a \defn{vp-bridge surface} for $(M,T)$ if $(M,T) \setminus H$ is the union of two distict vp-compressionbodies $(C_i, T_i)$ for $i =1,2$ such that $\boundary_+ C_1 = H = \boundary_+ C_2$. A bridge sphere in $S^3$ is an example of a vp-bridge surface, as is a Heegaard surface for the exterior of a spatial graph in a closed 3-manifold.

We now define the central tool of this paper: the multiple vp-bridge surface. See \cite{TT1} for more details. Informally, a multiple vp-bridge surface cuts $(M,T)$ up into vp-compressionbodies. More formally, a surface $\mc{H} = \mc{H}^+ \cup \mc{H}^-\subset (M,T)$ is a \defn{multiple vp-bridge surface} if each of $\mc{H}^+$ and $\mc{H}^-$ is the union of components with no component belonging to both, and $(M,T)\setminus \mc{H}$ is the union of vp-bridge compressionbodies $(C, T_C)$ such that $\mc{H}^+ = \bigcup \boundary_+ C$ and $\mc{H}^- \cup \boundary M = \bigcup \boundary_- C$. The components of $\mc{H}^+$ are called \defn{thick surfaces} and those of $\mc{H}^-$ are \defn{thin surfaces}. Given a multiple vp-bridge surface $\mc{H}$, we consider the dual graph where each component of $(M,T)\setminus \mc{H}$ is a vertex. Edges correspond to the components of $\mc{H}$. Equipping each component of $\mc{H}$ with a normal orientation makes the dual graph into the \defn{dual digraph} for $\mc{H}$. Suppose that $v = (C, T_C)$ is a vertex of the dual digraph with $e$ the edge corresponding to $\boundary_+ C$. We insist that the normal orientations are such that if $e$ is oriented into $v$, then the edges corresponding to the components of $\mc{H}^- \cap \boundary_- C$ (if any) are all oriented out of $v$ and if $e$ is oriented out of $v$, then the edges corresponding to the components of $\mc{H}^- \cap \boundary_- C$ are all oriented into $v$. If under such a constraint, the dual digraph is acyclic we call $\mc{H}$ an \defn{oriented multiple vp-bridge surface}\footnote{This definition is easily seen to be equivalent to that in \cite{TT2}.}. For a pair $(M,T)$, we let $\H(M,T)$ denote the set of oriented multiple vp-bridge surfaces for $(M,T)$ up to isotopy transverse to $T$. Observe that assigning a normal orientation to a vp-bridge surface makes it into an oriented multiple vp-bridge surface.


For $\mc{H} \in \H(M,T)$ we define the \defn{net extent} of $\mc{H}$ to be:
\[
\netextent(\mc{H}) = \extent(\mc{H}^+) - \extent(\mc{H}^-)
\]
and the \defn{net Euler characteristic} to be
\[
\netchi(\mc{H}) = -\chi(\mc{H}^+) + \chi(\mc{H}^-).
\]

If $x \geq 2\g(M) - 2$, then there exists a vp-bridge surface $H$ for $(M,T)$ with $-\chi(H) = x$. We say that such an $x$ is \defn{realizable}. For any admissible $x$, we can therefore define an invariant
\[
\netextent_x(M,T) = \min \netextent(\mc{H})
\]
where the minimum is taken over all $\mc{H} \in \H(M,T)$ such that $\netchi(\mc{H}) \leq x$. As a sequence indexed by $x$, $(\netextent_x(M,T))$ is non-increasing and is eventually constant at $\netextent_\infty(M,T)$. The paper \cite{TT2} thoroughly studied this invariant. In the next section, we review some of its properties. First we establish a basic lemma:
\begin{lemma}\label{knotlem}
Suppose $(M,T)$ is an irreducible pair such that $T$ is a link. If  $\mc{H}\in \H(M,T)$ then $\netextent(\mc{H})$ is an integer. In particular, for all realizable $x$, $\netextent_x(M,T)$ is an integer.
\end{lemma}
\begin{proof}
By definition, 
\[
2\netextent(\mc{H}) = -\chi(\mc{H}^+) + \chi(\mc{H}^-) + |\mc{H}^+ \cap T| - |\mc{H}^- \cap T|.
\]
Since each component of $\mc{H}$ is a closed surface, the terms involving Euler characteristic are even. Since $T$ is a link, each component of $T \setminus \mc{H}$ is a bridge arc, vertical arc, ghost arc or core loop. The bridge arcs have both endpoints on $\mc{H}^+$; the vertical arcs have one endpoint on $\mc{H}^+$ and one on $\mc{H}^-$ (and not on $\boundary M$); the ghost arcs have both endpoints on $\mc{H}^-$ (and not on $\boundary M$). Thus, the quantity $|\mc{H}^+ \cap T| - |\mc{H}^- \cap T|$ is also even.
\end{proof}

\section{Background on Thin Position and some technical results}\label{sec:thin position}

To understand the relationship between tunnel number and sums, we use thin position. The classical theory of thin position for knots in $S^3$ is originally due to Gabai \cite{Gabai}. It was applied to the study of spatial graphs in $S^3$ by Scharlemann and Thompson \cite{ST2}, who also adapted it to handle structures on 3-manifolds \cite{ST1}. In \cite{HS}, Hayashi and Shimokawa extended the theory to apply to knots in arbitrary 3-manifolds. In \cite{TT1}, we adapted Hayashi and Shimokawa's approach to spatial graphs in arbitrary 3-manifolds and extended the theory to handle sc-discs in full generality. The upshot is that we defined, for an irreducible pair $(M,T)$, a partial order on $\H(M,T)$, denoted $\thinsto$. If $\mc{H} \thinsto \mc{K}$, we say that $\mc{H}$ \defn{thins to} $\mc{K}$. If $\mc{K} \in \H(M,T)$ is minimal in the partial order (i.e. $\mc{K} \to \mc{J}$ implies $\mc{K} = \mc{J}$) we say it is \defn{locally thin}. We showed  \cite{TT1} that for every $\mc{H} \in \H(M,T)$, there exists a locally thin $\mc{K} \in \H(M,T)$ such that $\mc{H} \to \mc{K}$. In which case, $\netextent(\mc{H}) \geq \netextent(\mc{K})$ and $\netchi(\mc{H}) \geq \netchi(\mc{K})$ \cite{TT2}.

Furthermore, if $\mc{H}$ is locally thin, the following hold \cite{TT1}:
\begin{enumerate}
\item[(LT1)] Every  $H \cpt \mc{H}^+$, has the property that any two sc-discs for $H$ on opposite sides of $H$ and disjoint from $\mc{H}^-$ have boundaries that intersect. (That is, $H$ is \defn{sc-strongly irreducible}.)
\item[(LT2)] If $(C,T_C) \cpt (M,T) \setminus \mc{H}$ is a trivial product compressionbody, then $\boundary_- C \subset \boundary M$.
\item[(LT3)] $\mc{H}^-$ is c-essential in $(M,T)$.
\item[(LT4)] If there is an essential sphere in $(M,T)$ with 3 or fewer punctures then there is such a sphere in $\mc{H}^-$.
\item[(LT5)] If some component of $\mc{H}$ is an unpunctured sphere, then $T = \nil$ and $M$ is either a 3-ball or $S^3$.
\end{enumerate} 
Notice that (LT2) implies that if we surger $(M,T)$ along the twice and thrice-punctured spheres in $\mc{H}^-$, then each component of the resulting pair is either prime or trivial. None of the trivial components can be a handcuff curve as that would contradict irreducibility of $(M,T)$. We will also need a weaker version of (LT3):
\begin{enumerate}
\item[(wLT3)] $\mc{H}^-$ is c-incompressible in $(M,T)$.
\end{enumerate}

We will also need to know a little more about the effect of certain thin surfaces. The proof is a simpler version of the proofs in Section \ref{sec:vp-compressionbodies}.

\begin{theorem}\label{thin props}
Suppose that $(M,T)$ is a connected, irreducible, prime (3-manifold, graph) pair. Suppose that $\mc{H} \in \H(M,T)$ is locally thin. Then the following hold:
\begin{enumerate}
\item If some component of $\mc{H}^+$ is a torus with no punctures then $\mc{H}$ is connected. Also $M$ is $S^3$, $S^1 \times S^2$, a lens space, a solid torus or $T^2 \times I$. The graph $T$, if nonempty, is either a core loop or Hopf link (that is, the union of cores of solid tori on opposite sides of $\mc{H}$).
\item If some component $(C,T_C) \cpt (M,T) \setminus \mc{H}$ has $\delta(C,T_C) = 0$ and $-\chi(\boundary_+ C) = -\chi(\boundary_- C)$, then either $(C, T_C)$ is a trivial product compressionbody with $\boundary_- C \subset \boundary M$ or  $|\boundary_+ C \cap T| \geq 2$ and $T_C$ contains a vertex of $T$.
\end{enumerate}
\end{theorem}
\begin{proof}
 Suppose that some component $H$ of $\mc{H}^+$ is an unpunctured torus and $M$ is closed. Let $(C, T_C)$ and $(D, T_D)$ be the vp-compressionbodies of $(M,T)\setminus\mc{H}$ on either side of $H$.  Consider $(C, T_C)$. By definition, the compressionbody $C$ must be the result of removing some number (possibly zero) of open 3-balls from either a solid torus or $T^2 \times I$. Since $T \cap H = \nil$, each component of $\punct{T}_C$ is a ghost arc or a core loop. A leaf of the ghost arc graph corresponding to a spherical component of $\boundary \punct{C}$ would need to be incident to vertical arcs, so there are no such leaves. Similarly, an isolated vertex or a vertex of degree 2 cannot correspond to a sphere component of $\boundary \punct{T}_C$. Thus, $\boundary_- \punct{C}$ is either empty or a single torus, so $C$ is a solid torus or $T^2 \times I$. The graph $T_C$ is either a core loop or empty by Lemma \ref{ghostarcgraphlem}. By (LT 2), if $C= T^2 \times I$, then $\boundary_- C \subset \boundary M$. The same analysis holds for $(D, T_D)$ and Conclusion (1) follows.
 
Suppose now that $\delta(C, T_C) = 0$ for some $(C, T_C) \cpt (M,T)\setminus \mc{H}$. If $-\chi(\boundary_+ C) = -\chi(\boundary_- C)$, the compressionbody $C$ is a product $\boundary_- C \times I$. The ghost arc graph $\Gamma_C$ is acyclic by Lemma \ref{ghostarcgraphlem}. Suppose some component $\Gamma'$ contains an edge. Since $\Gamma'$ is a tree, it has at least two leaves, at most one of which can be $\boundary_- C$. Thus, some leaf of $\Gamma_C$ corresponds to a vertex of $T$. That vertex has degree at least 3, so is incident to at least two vertical arcs. Hence, $|\boundary_+ C \cap T| \geq 2$. Thus, (2) holds. If no component of $\Gamma_C$ contains an edge, then $\Gamma_C$ consists of isolated vertices. If there is more than one such vertex of $\Gamma_C$, then that vertex is also a vertex of $T_C$. It must be incident to at least 3 vertical arcs, as desired. By (LT2), we have Conclusion (2).
\end{proof}

We will also make use of the following lemma.
\begin{lemma}\label{crushing calc}
Suppose that $\mc{H} \in \H(M,T)$ satisfies (LT1), (LT2), (wLT3), (LT4), and (LT5) and that $\netextent(\mc{H}) = \netextent_x(M,T)$ for some $x \geq \netchi(\mc{H})$. Suppose that $F \cpt \mc{H}^-$ is a sphere with $p \geq 2$ punctures. Let $(\wihat{M}_1, \wihat{T}_1)$ and  $(\wihat{M}_2, \wihat{T}_2)$ be the result of surgering $(M,T)$ along $F$. Let $\mc{H}_i = \mc{H} \cap \wihat{M}_i$ for $i = 1,2$. Then all of the following hold:
\begin{enumerate}
\item $\mc{H}_1$ and $\mc{H}_2$ continue to satisfy (LT1), (LT2), (wLT3), (LT4), and (LT5).
\item $\netchi(\mc{H}) = \netchi(\mc{H}_1) + \netchi(\mc{H}_2) + 2$
\item $\netextent(\mc{H}) = \netextent(\mc{H}_1) + \netextent(\mc{H}_2) - (p-2)/2$
\item $\netextent_x(M,T) = \netextent_{x_1}(\wihat{M}_1, \wihat{T}_1) + \netextent_{x_2}(\wihat{M}_2, \wihat{T}_2) - (p-2)/2.$
where $x_i = \netchi(\mc{H}_i)$.
\end{enumerate}
\end{lemma}
\begin{proof}
The vertices of a graph are treated as negative boundary components of vp-compressionbodies, so $\mc{H}_1$ and $\mc{H}_2$ are multiple vp-bridge surfaces satisfying (LT1), (LT2), (wLT3), (LT4), and (LT5). (We note, however, that if $\mc{H}$ satisfies (LT3), $\mc{H}_1$ and $\mc{H}_2$ do not need to as $\mc{H}$ may have two thin surfaces that are parallel. If such is the case, then a thin surface can become $\boundary$-parallel after surgery.) The definition of net extent does treat vertices of the graph and boundary components of the 3-manifold differently from thin surfaces.  Hence, Conclusions (2) and (3) follow immediately from the fact that $\chi(F) = 2$ and $\extent(F) = (p-2)/2$. One inequality in Conclusion (4) can be proved by realizing that $\netextent(\mc{H}_i)$ is an upper bound for $\netextent_{x_i}(\wihat{M}_i, \wihat{T}_i)$. The other inequality can be proved as in  \cite[Theorem 5.5]{TT2}. We do not rely on it in this paper, so we omit the proof.
\end{proof}

Another technical lemma we will need can be found as Corollary 4.4 of \cite{TT2}. We restate it here for convenience.
\begin{lemma}\label{The delta lemma}
If $\delta(C,T_C) =0$, then $(\punct{C},\punct{T}_C)$ is one of the following:
\begin{enumerate}
\item[(VP1)] $(B^3, \text{ arc })$
\item[(VP2)] $(S^1 \times D^2, \nil)$
\item[(VP3)] $(S^1 \times D^2, \text{ core loop })$
\item[(VP4)] A vp-compressionbody such that each component of $\punct{T}_C$ is a vertical arc or ghost arc; there is no compressing or semicompressing disc for $\boundary_+ C$ in $(C, T_C)$; the ghost arc graph is connected, and the union of the ghost arcs with $\boundary_- \punct{C}$ is a spine for $\punct{C}$.
\end{enumerate}
\end{lemma}

In \cite{TT2}, we studied net extent and the related invariant ``width'' from the perspective of thin position. Our results concerning net extent can be summarized as follows:
\begin{theorem}[{Theorem 4.9 and Theorem 5.7 from \cite{TT2}}]\label{Add Thm} The following hold for net extent.

\begin{itemize}\spacing
\item (Unknot Detection)  Suppose that $(M,T)$ is a connected, irreducible pair with $T$ a knot or link and that every sphere in $M$ separates. Assume $\netextent_x(M,T) = 0$ for some realizable $x$. If there is no essential twice or thrice-punctured sphere in $(M,T)$, then either $M$ is $S^3$ or a lens space and $T$ is the unknot, a core loop, or a Hopf link, or $M$ is a solid torus and $T$ is a core loop.

\item (Additivity) Suppose $(M,T)$ is connected, irreducible, composite and that every sphere in $M$ separates. If $x$ is realizable for $(M,T)$, then there is a prime factorization of $(M,T)$ into pairs $(\wihat{M}_i, \wihat{T}_i)$ for $i = 1, \hdots, n$ such that for each $i$ there exists a realizable $x_i$ so that:
\begin{enumerate}
\item $x_1 + \cdots + x_n \leq x - 2(n-1)$, and
\item $\netextent_x(M,T) = \frac{-p_3}{2} + \sum\limits_{i=1}^n \netextent_{x_i}(\wihat{M}_i, \wihat{T}_i)$.
\end{enumerate}
where $p_3$ is the number of thrice-punctured summing spheres in the prime decomposition.
\end{itemize}
\end{theorem}

We also gave the following lower bound for net extent. Note that the right hand side is an integer or half integer.

\begin{corollary}[{ \cite[Cor. 4.8]{TT2}}]\label{Main Ineq}
For a pair $(M,T)$ such that every sphere in $M$ separates, we have
\begin{equation}
\netextent_\infty(M,T) \geq \frac{-\chi(\boundary M)}{4} + \frac{|\boundary M \cap T| - \chi(T)}{2}.
\end{equation}
\end{corollary}

We will need to understand when equality holds, or is close to holding, in the case when $T$ has genus 2. To that end, for $\mc{H} \in \H(M,T)$, define
\[
\Delta(\mc{H}) = 2\netextent(\mc{H}) - \big(\extent(\boundary M) + \frac{|\boundary M \cap T|}{2}  - \chi(T)\big).
\]
Recall that $\Delta(\mc{H})$ is non-negative and integral.
\begin{lemma}\label{Main Ineq2}
Suppose that $(M,T)$ is an irreducible pair. Let $\mc{H} \in \H(M,T)$. Then
\[
\Delta(\mc{H}) = \sum\limits_{(C, T_C)} \delta(C, T_C),
\]
where the sum is over all components $(C, T_C) \cpt (M,T)\setminus \mc{H}$.
\end{lemma}
\begin{proof}
Central to our proof of Corollary \ref{Main Ineq} was the observation that, for each $\mc{H} \in \H(M,T)$:
\begin{equation}
2\netextent(\mc{H}) - \extent(\boundary \punct{M}) = \sum\limits_{(C,T_C)} \delta(C,T_C).
\end{equation}
The sum is over all vp-compressionbodies $(C, T_C) \cpt (M,T)\setminus \mc{H}$. This follows from the realization that extent is additive under disjoint union and that the union of the positive boundaries of the components of $(\punct{M},\punct{T})\setminus H$ is equal to $\mc{H}^+$ and the union of the negative boundaries is the union of $\mc{H}^-$ with $\boundary \punct{M}$. We recall that $\boundary \punct{M}$ consists of both $\boundary M$ and spheres corresponding to the internal vertices of $T$. 

Thus,
\[
\sum\limits_{(C,T_C)} \delta(C,T_C) = 2\netextent(\mc{H}) - \extent(\boundary M) - \sum\limits_{v} \frac{\deg(v) - 2}{2}
\]
where the second sum is over all the \emph{internal} vertices of $T$.  Inserting $\pm |\boundary M \cap T|/2$ and computing the Euler characteristic produces what we want.
\end{proof}

\section{Analysis of vp-compressionbodies}\label{sec:vp-compressionbodies}

Throughout this section, assume that $(E, T_E)$ is a vp-compressionbody such that no component of $\boundary_- E$ is a sphere with two or fewer punctures. We consider several possibilities for the situation when $\boundary_+ E$ has small genus and few punctures.

\subsection{When $\boundary_+ E$ is a sphere with four or fewer punctures}

\begin{lemma}\label{sphere4}
Suppose that $\boundary_+ E$ is a sphere with four or fewer punctures. Then $(\punct{E}, \punct{T}_E)$ is one of the vp-compressionbodies pictured in Figure \ref{fig:FourPunctSphereCompBdy}. That is:
\begin{enumerate}
\item if $|H \cap T| = 2$, then $(E, T_E) = (\punct{E}, \punct{T}_E)$ is a trivial ball compressionbody;
\item if $|H \cap T| = 3$, then $(\punct{E}, \punct{T}_E)$ is a trivial product compressionbody.
\item if $|H \cap T| = 4$, then either $T_E$ is the union of two bridge arcs, $(\punct{E}, \punct{T}_E)$ is a trivial product compressionbody, or $\boundary_- \punct{E}$ is the union of two spheres and $\punct{T}_E$  is the union of a ghost arc and four vertical arcs.
\end{enumerate}
\end{lemma}
\begin{proof}
Assume that $\boundary_+ E$ is a sphere and that $|\boundary_+ E \cap T| \leq 4$. Since $\boundary_+ E$ is a sphere, by Corollary \ref{ghostarcgraph cor}, $\boundary_- E$ is the union of spheres and the corresponding ghost arc graph $\Gamma$ is acyclic. Since no component of $\boundary_- E$ is a sphere with two or fewer punctures, each isolated vertex of $\Gamma$  is incident to at least 3 vertical arcs of $\punct{T}_E$ and each leaf of $\Gamma$ is incident to at least two vertical arcs of $\punct{T}_E$. Thus, $(\punct{E}, \punct{T}_E)$ is one of the vp-compressionbodies pictured in Figure \ref{fig:FourPunctSphereCompBdy}.
\end{proof}

\begin{figure}[ht!]
\labellist
\small\hair 2pt
\pinlabel {(1)} [t] at 70 35
\pinlabel {(2)} [t] at 222 35
\pinlabel {(3)} [t] at 370 35
\pinlabel {(3)} [t] at 518 35
\pinlabel {(3)} [t] at 669 35
\endlabellist
\centering
\includegraphics[scale=0.4]{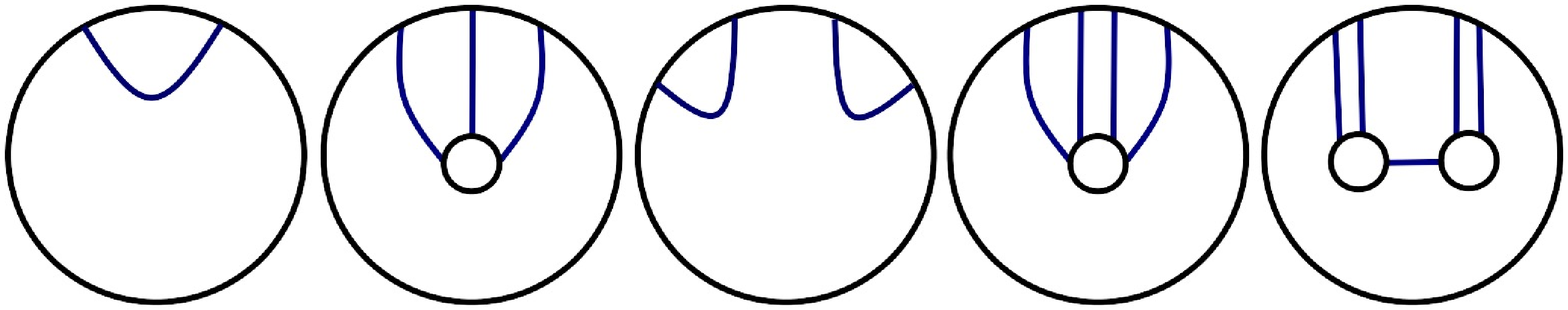}
\caption{The vp-compressionbodies $(E, T_E)$ with $\boundary_+ E$ a sphere with 4 or fewer punctures and no component of $\boundary_- E$ a sphere with two or fewer punctures.}
\label{fig:FourPunctSphereCompBdy}
\end{figure}

\subsection{When $\boundary_+ E$ is a torus with 2 or fewer punctures}

\begin{lemma}\label{toruscompbody}
Suppose that $\boundary_+ E$ is a torus such that $|\boundary_+ E \cap T| \leq 2$. Then $(\punct{E}, \punct{T}_E)$ is one of the vp-compressionbodies pictured in Figure \ref{fig:TwoPunctTorusCompBdy}. That is, one of the following holds:
\begin{enumerate}
\item $(E, T_E)$ is a trivial product compressionbody with 0, 1, or 2 vertical arcs.
\item $E = T^2 \times I$ and $T_E$ is a bridge arc.
\item $\punct{E}$ is obtained by removing an open 3-ball from $T^2 \times I$ and $\punct{T}_E$ is the union of a ghost arc joining the torus component of $\boundary_- \punct{E}$ to the spherical component of $\boundary_- \punct{E}$ with two vertical arcs, each joining the spherical component of $\boundary_- \punct{E}$ to $\boundary_+ E$.
\item $E$ is a solid torus and $T_E$ is a bridge arc
\item $(E, T_E)$ is (solid torus, core loop) or (solid torus, $\nil$)
\item $\punct{E}$ is obtained by removing an open 3-ball from a solid torus and $\punct{T}_E$ is the union of a ghost arc and 1 or 2 vertical arcs.
\item $\punct{E}$ is obtained by removing two open 3-balls from a solid torus and $\punct{T}_E$ is the union of two ghost arcs, each joining the two components of $\boundary_- \punct{E}$, with two vertical arcs. Each component of $\boundary_- \punct{E}$ is incident to one vertical arc.
\item $\punct{E}$ is obtained by removing two open 3-balls from a solid torus and $\punct{T}_E$ is the union of two ghost arcs and two vertical arcs. One of the ghost arcs joins the two spherical components of $\boundary_- \punct{E}$. The other has both endpoints at the same component $P \cpt \boundary_- \punct{E}$. The two vertical arcs each have an endpoint on $\boundary_- \punct{E} \setminus P$.\end{enumerate}
\end{lemma}

\begin{figure}[ht!]
\labellist
\small\hair 2pt
\pinlabel {(1)} [t] at 124 344
\pinlabel {(1)} [t] at 258 344
\pinlabel {(1)} [t] at 394 344
\pinlabel {(2)} [t] at 124 232
\pinlabel {(3)} [t] at 251 232
\pinlabel {(4)} [t] at 396 232
\pinlabel {(5)} [t] at 58 117
\pinlabel {(5)} [t] at 189 117
\pinlabel {(6)} [t] at 329 117
\pinlabel {(6)} [t] at 461 117
\pinlabel {(7)} [r] at 83 42
\pinlabel {(8)} [l] at 430 40
\endlabellist
\centering
\includegraphics[scale=0.5]{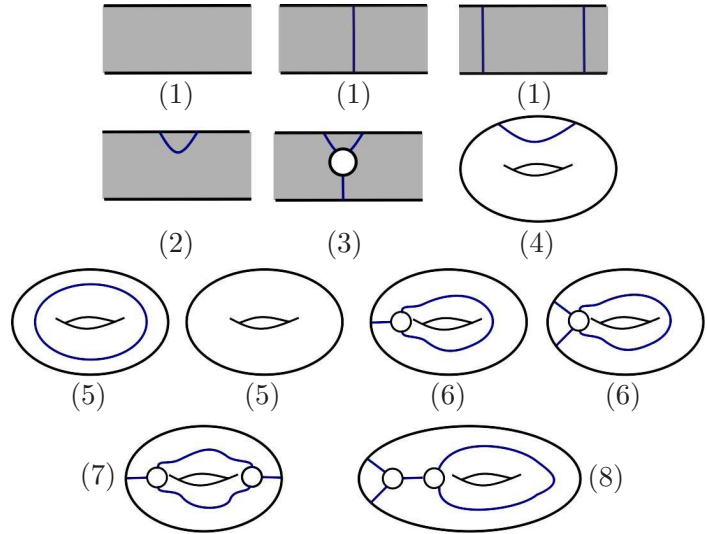}
\caption{The vp-compressionbodies $(E, T_E)$ with $\boundary_+ E$ a torus with 2 or fewer punctures and no component of $\boundary_- E$ a sphere with two or fewer punctures. The shaded rectangles indicate $T^2 \times I$ and the hollow circles denote spherical components of $\boundary_- \punct{E}$.}
\label{fig:TwoPunctTorusCompBdy}
\end{figure}

\begin{proof}
Suppose that $H = \boundary_+ E$ is a torus with $|H \cap T_E| \leq 2$. The case when $|H \cap T_E| = 0$ is covered by Theorem \ref{thin props}. By Corollary \ref{ghostarcgraph cor}, $\boundary_- \punct{E}$ is the union of spheres and at most one torus. Also, the ghost arc graph $\Gamma$ has at most one cycle. If $\boundary_- \punct{E}$ contains a torus, then $\Gamma$ is acyclic. Suppose that $v$ is a spherical vertex of $\Gamma$. If $v$ is isolated, then it must be incident to at least 3 vertical arcs of $\punct{T}_E$. Since $|\boundary_+ E \cap T_E| \leq 2$, this is impossible. Thus, no spherical component of $\boundary_- \punct{E}$ is an isolated vertex of $\Gamma$. If $v$ is a leaf of $\Gamma$, then it must be incident to at least two vertical arcs. Thus, at most one leaf of $\Gamma$ is a spherical component of $\boundary_- \punct{E}$. 

We conclude that if $\boundary_- E$ contains a torus, then $\boundary_- \punct{E}$ contains at most one sphere. If $\boundary_- \punct{E}$ contains a torus but does not contain any spheres, then $(E, T_E)$ is a trivial product compressionbody or $E = T^2 \times I$ and $T_E$ is a bridge arc. This is Conclusion (1) or (2).  If it contains a torus and one sphere, then there is a ghost arc joining the two components of $\boundary_- \punct{E}$, as the sphere cannot be incident to three vertical arcs. The spherical component is incident to two vertical arcs of $\punct{T}_E$ and the torus component is not incident to any vertical arcs. This is Conclusion (3).  

If $\boundary_- \punct{E} = \nil$, then $E$ is a solid torus and $T_E$ is  empty, a core loop or bridge arc, giving Conclusion (4) or (5). Assume, therefore, that $\boundary_- \punct{E}$ is the nonempty union of spheres.  By our previous remarks, each vertex of $\Gamma$ must have degree at least 2.  Since $\Gamma$ does not have isolated vertices, it must contain an edge. Since it has at most one leaf, it must contain a cycle. By our previous remarks, $\Gamma$ contains a unique cycle and at most one vertex not in the cycle. 

Suppose that $v$ is a vertex of $\Gamma$ belonging to the cycle. If $v$ has degree 2, then the corresponding spherical component of $\boundary_- \punct{E}$ must be incident to at least one vertical arc. Thus, if $\Gamma$ is a cycle, then it contains at most two vertices. If it contains a single vertex, then we have Conclusion (5). If $\Gamma$ is a cycle with two vertices, then we have Conclusion (6). 

Finally, suppose that $\Gamma$ contains a vertex $v$ not in the cycle. That vertex $v$ must be the unique such vertex and must be incident to two vertical arcs. Thus, there can be no other vertical arcs. We arrive at Conclusion (7) or (8).
\end{proof}

\subsection{When $\boundary_+ E$ is an unpunctured genus 2 surface}

The case when $\boundary_+ E$ is a genus two surface disjoint from $T$ is quite simple.

\begin{lemma}\label{genus2disjt}
Suppose that $\boundary_+ E$ is a genus two surface disjoint from $T_E$. Then one of the following occurs:
\begin{enumerate}
\item $E$ is a genus 2 handlebody and $T_E$ is either a knot or 2-component link contained in a spine for $E$.
\item $(\punct{E}, \punct{T}_E)$ is obtained from a (genus 2 handlebody, spine) pair by puncturing the vertices.
\item $\boundary_- E$ is a single torus and $T_E$ is empty or a ghost arc
\item $\boundary_- \punct{E}$ is the union of a torus and a thrice punctured sphere. $\punct{T}_E$ is the union of a ghost arc joining the components of $\boundary_- \punct{E}$ and a ghost arc with both ends on the spherical component of $\boundary_- \punct{E}$.
\item $\boundary_- E$ is the union of two tori and $T_E$ is either empty or a ghost arc joining the tori.
\item $(E, T_E)$ is a trivial product compressionbody with $T_E =\nil$.
\end{enumerate}
\end{lemma}
\begin{proof}
Let $\Gamma$ be the ghost arc graph for $(E, T_E)$. If $v$ is an isolated vertex, degree 1 vertex, or degree 2 vertex corresponding to a spherical component of $\boundary_- \punct{E}$ then it must be incident to at least one vertical arc. Such an arc would mean that $\boundary_+ E \cap T_E \neq \nil$, a contradiction. Thus, there are no such vertices. Since there can also be no bridge arcs in $T_E$, $\punct{T}_E$ is the union of ghost arcs. Let $S$ be the frontier of a regular neighborhood of $\boundary_- \punct{E} \cup \punct{T}_E$

By Lemma \ref{ghostarcgraphlem}, the genus of $\boundary_+ E$ is at least the genus of $S$. Thus, the genus of $S$ is at most 2. If $S$ is empty, we have Conclusion (1). Assume, therefore, that $S \neq \nil$. If $\boundary_- E = \nil$, we have Conclusion (2). If $\boundary_- E$ contains a single torus, we have Conclusion (3) or (4). If $\boundary_- E$ contains two tori, we have Conclusion (5). If $\boundary_- E$ is a genus two surface, then we have Conclusion (6).
\end{proof}

The next proposition follows immediately by applying Lemma \ref{genus2disjt} to the vp-compressionbodies on either side of a vp-bridge surface $H$.

\begin{proposition}\label{genus2disjtprop}
Suppose that $M$ is closed, $T$ is connected and $H \in \H(M,T)$ is a genus two surface disjoint from $T$, then $H$ is a Heegaard surface for $M\setminus T$ and either $T$ is a knot of tunnel number at most 1 or $T$ is a genus 2 graph with handlebody exterior.
\end{proposition}

\section{Types of graphs of low net extent}\label{sec:pairs of low extent}

The results of this section roughly correspond to the task in knot theory of understanding knots having either bridge spheres with few punctures (i.e. the unknot and 2-bridge knots) or a genus 1 bridge surface with few punctures (i.e. the (1,1) knots).  We begin by defining the knot and graph types relevant to our investigation and then show that none of them are Brunnian $\theta$-graphs.

\subsection{Special examples}
We begin with some special classes of genus 2 spatial graphs. It turns out that they all have $\netextent_\infty(M,T) \leq 1$. We adapt the notation ``$(g,b)$-curve'' from knot theory \cite{Doll}, which in that context means that a knot has a genus $g$ bridge surface intersecting the knot in $2b$ points and the pair $(g,b)$ is minimal is some ordering on such bridge surfaces. (We do not need the complete definition.) Throughout we assume that $(M,T)$ is a nontrivial irreducible connected pair with $T$ a knot, link or genus 2 graph.

The pair $(M,T)$ is a \textbf{(lens space, core loop) pair} or \textbf{(1,0)-curve} if $M$ is a lens space $(\neq S^3, S^1 \times S^2)$ and $T$ is a core loop with respect to a Heegaard torus $H$ for $M$. Note that $\netextent(H) = 0 = \netchi(H)$. The pair $(M,T)$ with $T$ a handcuff graph is a \defn{Hopf graph}, if $M$ is closed and there is a torus $H \in \H(M,T)$ intersecting $T$ in a single point (necessarily in the separating edge of $T$). Note that $\netextent(H) = 1/2$ and $\netchi(H) = 0$. See Figure \ref{hopfgraph}. 

\begin{figure}
\includegraphics[scale=0.35]{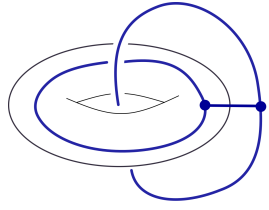}
\caption{The Hopf graph in $S^3$ or a lens space. The torus is a genus 1 Heegaard surface.}
\label{hopfgraph}
\end{figure}

If there exists a sphere $H \in \H(M,T)$ such that $|H \cap T| = 4$, then $H$ is a \defn{2-bridge sphere}. Note that $\netextent(H) = 1$ and $\netchi(H) = -2$. A Hopf graph in $S^3$ admits a 2-bridge sphere, as does the trivial 2-bouquet. So we define a pair $(M,T)$ to be \defn{2-bridge} or a \defn{(0,2)-curve} if it is not a Hopf graph or trivial 2-bouquet and yet admits a 2-bridge sphere. Schematic depictions of the two types of 2-bridge genus 2 graphs are shown in Figure \ref{2bridge}.  It turns out that if a 2-bouquet has a 2-bridge sphere, then it is trivial, so there are no 2-bridge trivial 2-bouquets.

\begin{figure}
\includegraphics[scale=0.3]{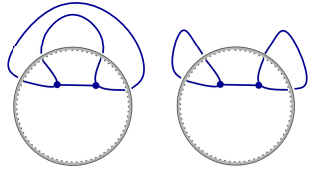}
\caption{Two-bridge genus 2 graphs}
\label{2bridge}
\end{figure}

If the pair $(M,T)$ is neither trivial nor 2-bridge but $M$ is closed and there exists a twice-punctured torus $H \in \H(M,T)$, then $(M,T)$ is a \defn{(1,1)-curve}. Note that $\netextent(H) = 1$ and $\netchi(H) = 0$. Figure \ref{oneonegraphfig} depicts the three kinds of genus 2 graphs with (1,1)-bridge surfaces.

If $M$ is neither a trivial, 2-bridge, Hopf graph, or  (1,1)-curve but there exists an unpunctured genus 2 surface $H \in \H(M,T)$ (equivalently, a genus 2 Heegaard surface for $M \setminus T$) then $(M,T)$ is a \defn{(2,0)-curve}. Note that $\netextent(H) = 1$ and $\netchi(\mc{H}) = 2$. For convenience, say that $(M,T)$ is \defn{knotted of low complexity}, if it is a (1,0)-curve, (0,2)-curve, a (1,1)-curve, or a (2,0)-curve. 

\begin{figure}
\centering
\includegraphics[scale=0.5]{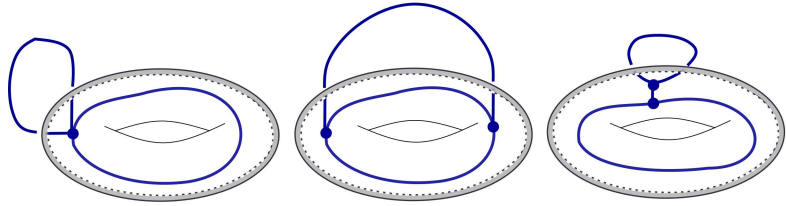}
\caption{Depictions of  (1,1) 2-bouquets, $\theta$-curves, and handcuff curves. The dark band indicates that there may be braiding arising from the homeomorphism gluing the two solid tori together.}
\label{oneonegraphfig}
\end{figure}

We have two more classes of knots and graphs to define. The first is a knot which we call a ``propeller knot.'' There are two types of propeller knot, both pictured in Figure \ref{fig:propellerknot}. As with our other examples it is defined using a multiple vp-bridge surface, although now the multiple vp-bridge surface is disconnected. Suppose that $(M,T)$ is an irreducible pair with $T$ a knot  and $M$ closed. Suppose that $(M,T)$ admits an oriented multiple vp-bridge surface $\mc{H}$ such that $\mc{H}^+$ is the union of two unpunctured genus 2 surfaces and $\mc{H}^-$ is either a single twice-punctured torus or two once-punctured tori. Observe that $\netchi(\mc{H}) = 4$ and $\netextent(\mc{H}) = 1$. If $(M,T)$ is nontrivial and is not knotted of low complexity, then $(M,T)$ is a \defn{propeller knot}. The surface $\mc{H}$ is the \defn{standard propeller surface}.

\begin{figure}
\labellist
\small\hair 2pt
\pinlabel $T^2$ [r] at 35 154
\pinlabel $T^2$ [r] at 258 154
\pinlabel $T^2$ [l] at 423 154
\endlabellist
\centering
\includegraphics[scale=0.5]{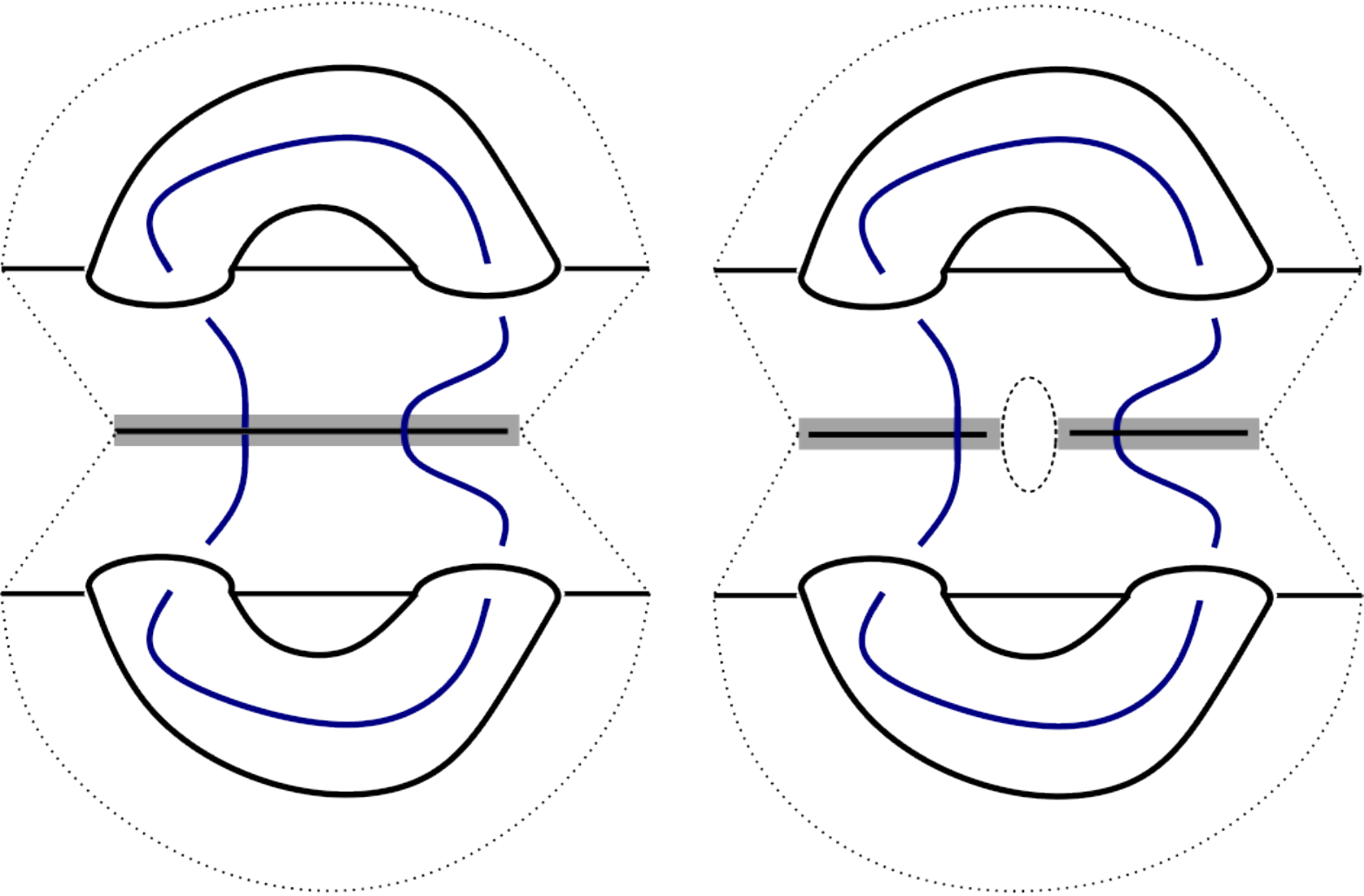}
\caption{Schematic depictions of two types of propeller knot. On the left is shown the type where $\mc{H}^-$ is a twice punctured torus and on the right is shown the type where $\mc{H}^-$ is two once punctured tori. The dark regions are a reminder that the gluing map may be complicated. The dashed lines enclose the 3-manifolds; for instance, there is a genus 2 handlebody above the upper genus 2 surfaces.}
\label{fig:propellerknot}
\end{figure}

Our final class of spatial graphs is the class of ``Hopf slinkies.'' To define them, we need to begin by considering some higher genus spatial graphs. They are depicted in Figures \ref{hopfringlet} and \ref{hopfified}. Note that each has at least one vertex of degree 4.

Suppose that $(M,T')$ is a pair such that $T'$ is connected, $M$ is closed, and $\boundary \eta(T')$ is a genus 3 surface. Suppose also that there exists a torus $H \in \H(M,T)$ such that $|H \cap T| = 2$. If $T$ has two vertices, two loops one based at each vertex, and two edges joining the two vertices and if $H$ is disjoint from the two loops, then $(M,T)$ is a \defn{Hopf ringlet}. See Figure \ref{hopfringlet}. Suppose that $T$ has three vertices $v_1, v_2, v_3$ (the labelling is immaterial). If there is a loop based at $v_1$, two edges joining $v_2$ to $v_3$, and an edge joining $v_1$ to each of $v_2$ and $v_3$ and if $H$ intersects both of the latter two edges, then $(M,T)$ is a \defn{Hopfified $\theta$-curve}. If there are loops based at $v_1$ and $v_3$, a single edge joining $v_2$ and $v_3$, and two edges joining $v_1$ to $v_2$ and if $H$ intersects both of these latter two edges, then $(M,T)$ is a \defn{Hopfified handcuff curve}. See Figure \ref{hopfified}. In all three cases, we call $H$ the \defn{associated torus}. Observe that since in all three cases, there is a 2-component sublink of $T$ of linking number 1, no pair in these three classes of spatial graphs can be trivial. 

\begin{figure}
\centering
\includegraphics[scale=0.5]{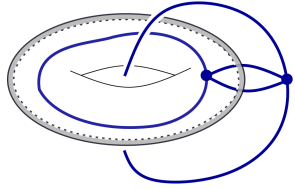}
\caption{A Hopf ringlet}
\label{hopfringlet}
\end{figure}

\begin{figure}
\centering
\includegraphics[scale=0.5]{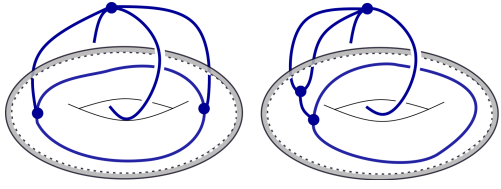}
\caption{On the left is a hopfified $\theta$-curve and on the right is a hopfified handcuff curve. In both cases, the torus is twice-punctured and the grey band indicates some homeomorphism of the twice-punctured torus to itself. It is suggestive that the exterior of both is a genus 3 handlebody.}
\label{hopfified}
\end{figure}

A pair $(M,T)$ with $T$ a knot, $\theta$-curve, handcuff curve, or 2-bouquet is a \defn{Hopf slinky} if it is a 4-valent vertex sum of the form:
\[
(M_1, T_1) \#_4 (M_2, T_2) \#_4 \cdots \#_4 (M_p, T_p)
\]
for $p \geq 2$, such that:
\begin{itemize}
\item For each $i < p$, $(M_i, T_i)$ and $(M_{i+1}, T_{i+1})$ are 4-valent vertex summed
\item $(M_1, T_1)$ is either a Hopfified $\theta$-curve, a Hopfified handcuff curve, Hopf ringlet, (1,1)-curve that is a 2-bouquet, or (2,0)-curve that is a 2-bouquet 
\item $(M_p, T_p)$ is either a (1,1)-curve or a (2,0)-curve that is a 2-bouquet. 
\item Each $(M_i, T_i)$ for $1 < i < p$ is a Hopf ringlet.
\end{itemize}
The factorization is called the \defn{slinky factorization}. See Figure \ref{fig: hopfslinky} for an example. If all the 4-punctured spheres in $(M,T)$ arising from the vertex sums are essential in $(M,T)$, we say that $(M,T)$ is an \defn{essential Hopf slinky}. The pairs $(M_1, T_1)$ and $(M_p, T_p)$ are the \defn{ends of the slinky}. From the definition, we can construct a multiple vp-bridge surface $\mc{H}$ for a Hopf slinky $(M,T)$ where the 4-punctured spheres corresponding to the 4-valent vertex sums comprise $\mc{H}^-$ and  in $(M_i, T_i)$ for $i \neq 1,p$, $\mc{H} \cap M_i$ is a twice-punctured torus. The surface $\mc{H} \cap (M_i, T_i)$ for $i = 1,p$ is a twice-punctured torus or unpunctured genus 2 surface. Such a multiple vp-bridge surface is called the \defn{standard slinky surface} for the slinky factorization. Note that $\netextent(\mc{H}) = 1$. We define the \defn{length} $\ell(\sigma)$ of the Hopf slinky $\sigma$ to be the minimum of $\netchi(\mc{H})$ over all standard slinky surfaces for $\sigma$; it is an even integer. Note that if $p$ is the number of factors in a slinky factorization of minimal length, then
\[
2p + 2\geq \ell(\sigma) \geq  2p-2 \geq 2.
\]

\begin{figure}[ht!]
\labellist
\small\hair 2pt
\pinlabel $\tau$ at 347 89
\endlabellist
\includegraphics[scale=0.5]{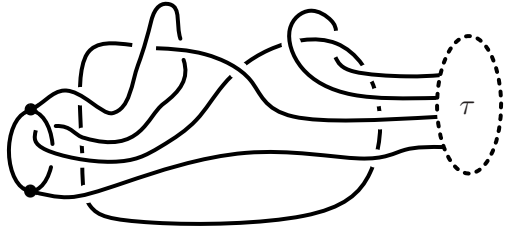}
\caption{To create a simple example of a Hopf slinky of length 4, insert a 2-tangle of Heegaard genus 2 into the ball marked $\tau$ so that the result is a $\theta$-curve. Since the 4-valent vertex sum depends on a choice of spherical 4-braid and the cycle containing both vertices of a Hopf ringlet can be any (1,1) knot, Hopf slinkies can be much more complicated than this.}
\label{fig: hopfslinky}
\end{figure}

We will ultimately prove that these graphs characterize genus 2 graphs with net extent 1. The next two sections are taken up with this task. First, however, we consider whether or not the classes of graph introduced in this section can be Brunnian.

\subsection{A note on Brunnian graphs}

The Kinoshita graph is an example of a $\theta$-curve $T \subset S^3$ such that there exists a sphere $P \in \H(S^3, T)$ that separates the vertices of $T$ and $|P \cap T| = 5$, \cite{Scharlemann-Brunnian}. The sphere $P$ cuts off a bridge arc from one of the edges of $T$. Meridionally stabilizing $P$ along that bridge arc produces a torus $H \in \H(S^3,T)$ that separates the vertices of $T$ and where $|H \cap T| = 3$. Thus, if all we know is that a nontrivial $\theta$-curve $T$ in $S^3$ has a genus 1 bridge surface $H$, with $|H \cap T| = 3$, we cannot conclude from those facts alone that $T$ has a knotted cycle. It would be interesting to classify all Brunnian $\theta$-curves having a genus 1 bridge surface $H$ intersecting the graph in three points. 

On the other hand, it turns out that a genus 2 graph with a bridge sphere having four or fewer punctures, or a vp-bridge torus with two or fewer punctures cannot be Brunnian. We'll use the following theorem of Ozawa and Tsutsumi \cite{OT}, though not in full generality. The version for genus 2 graphs is not difficult to prove directly.

\begin{theorem}[Ozawa-Tsutsumi]\label{thm: OT}
Suppose that $(S^3,T)$ is a pair such that $T$ is abstractly planar, nontrivial, and for every proper subgraph $T' \subset T$, the pair $(S^3,T')$ is trivial. Then the exterior of $T$ in $S^3$ is irreducible and $\boundary$-irreducible. In particular, the exterior of $T$ is not a handlebody.
\end{theorem}

\begin{corollary}\label{smallBrun}
Suppose that $(S^3,T)$ is an irreducible pair with $T$ a genus 2 graph. If $(S^3,T)$ is trivial, a $(0,2)$-curve, $(1,1)$-curve, $(2,0)$-curve, or Hopf graph then the exterior of $T$ is a genus 2 handlebody. In particular, $T$ is not a Brunnian $\theta$-curve.
\end{corollary}
\begin{proof}
If $M \setminus T$ has a genus 2 Heegaard surface $H'$, then the result follows immediately from Proposition \ref{genus2disjtprop} and Theorem \ref{thm: OT}. We will show that this is the case in each situation.

The cases when $(M,T)$ is a (2,0)-curve or trivial are immediate. Suppose therefore $H$ is a vp-bridge surface for $(M,T)$ realizing the fact that it is a $(0,2)$-curve, $(1,1)$-curve or Hopf curve. Choose a side of $H$ in $M$ and tube $H$ along all bridge arcs of $T \setminus H$ on that side, obtaining $H' \in \H(M,T)$.  If $(M,T)$ is a (0,2)-curve or (1,1)-curve, then one side of $H$ contains only bridge arcs and so we can construct a genus two unpunctured surface $H' \in \H(M,T)$, as desired. For a Hopf graph, there are no bridge arcs, but if we remove a regular neighborhood of one loop and then a regular neighborhood of the vertical arc on that side, we again construct a genus two unpunctured $H' \in \H(M,T)$.
\end{proof}

We also need the fact that Hopf slinkies are not Brunnian.

\begin{proposition}\label{no Brunnian slinky}
Suppose that $(M, T)$ is an essential Hopf slinky with $T$ a $\theta$-curve and that it is of the form 
\[
(M_1, T_1) \#_4 (M_2, T_2) \#_4 \cdots \#_4 (M_p, T_p)
\]
as in the definition. Then $T$ is not Brunnian.
\end{proposition}

\begin{proof}
Suppose $T$ is a $\theta$-curve, so that $(M_1, T_1)$ is a Hopfified $\theta$-curve. Assume, for a contradiction, that $T$ is Brunnian. Let $F$ be the 4-punctured sphere such that surgering $(M,T)$ along $F$ produces $(M_1, T_1)$ as one of the components. Let $\lambda_0$ and $\lambda_1$ be the two cycles of $T$ containing the edge that intersects $F$. Since $T$ is Brunnian, both are unknots. Consequently, $F$ is compressible in both $(M,\lambda_0)$ and $(M,\lambda_1)$. In fact, it must be compressible to the side $W$ containing the vertices of $T$. Since no edge of $T$ contains a local knot, both $(W, \lambda_0 \cap W)$ and $(W, \lambda_1 \cap W)$ are trivial (i.e. rational) tangles. Recall the existence of the disc whose boundary is a cycle of $T$ and that is once-punctured by $e$. Thus, if we glue to $\boundary W = F$ another trivial tangle $(B, \tau)$, we can produce links $(S^3, \tau \cup \lambda_0)$ and $(S^3, \tau \cup \lambda_1)$ having different linking numbers. Consequently, there is no homeomorphism of pairs taking $(W, \lambda_0)$ to $(W, \lambda_1)$ that fixes $F \cap \lambda_0$. Thus, $(M, \lambda_0)$ and $(M,\lambda_1)$ are obtained by attaching the rational tangles $(W, \lambda_0 \cap W)$ and $(W, \lambda_1 \cap W)$ to the prime tangle $(S^3 \setminus W, T \setminus W)$. However, by \cites{BS1, BS2} (see \cite[Theorem 4]{EM}), it is impossible to attach two inequivalent rational tangles to a prime tangle and arrive at $S^3$ in both instances.
\end{proof}

\section{Results about knots and graphs of small net extent}\label{sec:small extent}

The next theorem generalizes \cite[Theorem 7.5]{TT2}.

\begin{theorem}\label{low ne class}
Suppose that $(M,T)$ is a connected, irreducible, noncomposite pair such that every sphere in $M$ separates, $M$ is closed, and $T$ is a knot or genus 2 graph. Suppose that $\mc{H} \in \H(M,T)$ satsifies (LT1), (LT2), (wLT3), (LT4), and (LT5). Then
\begin{enumerate}
\item $\netextent(\mc{H}) = 0$ if and only if $(M,T)$ is a a trivial knot or a (1,0)-curve and $\mc{H}$ is a 2-punctured sphere or unpunctured torus, respectively.
\item $\netextent(\mc{H}) = 1/2$ if and only if $(M,T)$ is a trivial $\theta$-curve or Hopf graph and $\mc{H}$ is a 3-punctured sphere or once-punctured torus, respectively.
\item  For $T$ a genus 2 graph, $\netextent(\mc{H}) = 1$ if and only if $(M,T)$ is either knotted of low complexity, an essential Hopf slinky, or trivial 2-bouquet, and $\mc{H}$ is a 4-punctured sphere, 2-punctured torus, unpunctured genus 2 surface, or standard slinky surface. 
\item For $T$ a knot, $\netextent(\mc{H})= 1$ if and only if $(M,T)$ is either knotted of low complexity, a propeller knot, or an essential Hopf slinky, and $\mc{H}$ is a 4-punctured sphere, 2-punctured torus, unpunctured genus 2 surface, standard propeller surface, or standard slinky surface.
\end{enumerate}
\end{theorem}

\begin{proof}
Suppose that $(M,T)$ is connected and irreducible with $M$ closed and $T$ a knot or genus 2 graph. Suppose that $(M,T)$ is either trivial or prime. In either case, the important point is that there does not exist an essential twice or thrice-punctured sphere in $(M,T)$. Let $\mc{H} \in \H(M,T)$ be locally thin. Set $x = \netchi(\mc{H})$.

One direction of each biconditional is clear. It remains to establish the other directions of the biconditionals. Assume, therefore, that $\netextent(\mc{H}) \leq 1$.

 By (LT5), no component of $\mc{H}$ is an unpunctured sphere. By Theorem \ref{thin props}, if some component of $\mc{H}^+$ is an unpunctured torus, then $(M,T)$ is a trivial knot or (1,0)-curve.  If some component of $\mc{H}^+$ is a 2-punctured sphere, then by Lemma \ref{sphere4} applied to the vp-compressionbodies on either side of the sphere, Conclusion (1) holds. Henceforth, assume that no component of $\mc{H}^+$ is an unpunctured torus or a sphere with two or fewer punctures. Consequently, by Lemma \ref{The delta lemma}, whenever $(C, T_C) \cpt (M,T) \setminus \mc{H}$ has $\delta(C, T_C) = 0$, it must be of Type (VP4). In particular, if $\boundary_- C = \nil$, then $T_C$ contains a vertex of $T$. 
 
Since the dual digraph to $\mc{H}$ is acyclic, it has at least one source and at least one sink. Sources and sinks correspond exactly to the components $(C, T_C) \cpt (M,T)\setminus \mc{H}$ with $\boundary_- C = 0$. We will refer to such vp-compressionbodies as \defn{leaves} of the dual digraph. If a leaf $(C, T_C) \cpt (M,T)\setminus \mc{H}$ of the dual digraph has $\delta(C, T_C) = 0$, we observe that it must contain a vertex. 
 
Recall, 
\[
\Delta(\mc{H}) = 2\netextent(\mc{H}) + \chi(T)
\]
since $M$ is closed. If $T$ is a knot, by Lemma \ref{knotlem}, $\Delta(\mc{H}) \in \{0, 2\}$. If $T$ is a genus 2 graph, then $\Delta(\mc{H}) \in \{0,1\}$. By  Lemma  \ref{Main Ineq2}, 
\[
\sum\limits_{C, T_C} \delta(C, T_C)= \Delta(H) \leq 2.
\]
where the sum is over all components $(C, T_C) \cpt (M,T) \setminus \mc{H}$. Recall that $\delta(C, T_C)$ is a non-negative integer. Consequently, there are at most two such $(C, T_C)$ with $\delta(C, T_C) = 1$. 

Thus, we observe the following:
\begin{itemize}
\item If $\netextent(\mc{H}) = 0$, then $T$ is a knot  or (1,0)-curve and $\mc{H}$ is either a twice-punctured sphere or an unpunctured torus.
\item If $\netextent(\mc{H} )= 1/2$, then $T$ is a genus 2 graph and every vp-compressionbody of $(M,T)\setminus\mc{H}$, including the leaves of the dual digraph, are of Type (VP4). 
\item If $\netextent(\mc{H}) = 1$ and $T$ is a knot, there exist exactly two leaves of the dual digraph, one a sink and one a source, and each leaf has $\delta = 1$. Every other vp-compressionbody of $(M,T)\setminus \mc{H}$ is of Type (VP4).
\item If $\netextent(\mc{H}) = 1$ and $T$ is a graph, exactly one vp-compressionbody of $(M,T)\setminus \mc{H}$ has $\delta = 1$ and all the others (including at least one leaf of the dual digraph) are of Type (VP4).
\end{itemize}

Henceforth, assume that $(M,T)$ is not a trivial knot or (1,0) curve. We start by showing that either one of Conclusions (2), (3), or (4) hold or:
\begin{enumerate}
\item[(5)] There exists a 4-punctured sphere $F \cpt \mc{H}^-$ such that surgery along $F$ results in two connected pairs $(\wihat{M}_1, \wihat{T}_1)$ and $(\wihat{M}_2, \wihat{T}_2)$ such that $(\wihat{M}_1, \wihat{T}_1)$ is either a (1,1)-curve that is a 2-bouquet, a (2,0)-curve that is a 2-bouquet, a Hopf ringlet, a Hopfified $\theta$-curve, or a Hopfified handcuff curve. The pair $(\wihat{M}_2, \wihat{T}_2)$ is a 2-bouquet. Furthermore,  unless $(\wihat{M}_1, \wihat{T}_1)$ is a (1,1)-curve or (2,0)-curve that is a 2-bouquet, $(\mc{H} \cap \wihat{M}_1) \in \H(\wihat{M}_1, \wihat{T}_1)$ is a twice-punctured torus. If it is a (1,1)-curve or (2,0)-curve that is a 2-bouquet, then $(\mc{H} \cap \wihat{M}_1) \in \H(\wihat{M}_1, \wihat{T}_1)$ is either a twice-punctured torus or an unpunctured genus 2 surface.
\end{enumerate}

We will then show that (5) can be applied inductively to construct a Hopf slinky. 

\textbf{Case 1:} $\netextent(\mc{H}) = 1/2$ or $\netextent(\mc{H}) = 1$ and $T$ is a graph.

Let $(C, T_C)$ and $(C', T'_C)$ be distinct components of $(M,T)\setminus \mc{H}$ such that $(C, T_C)$ is a leaf of the dual digraph and $\boundary_+ C' = \boundary_+ C$. Call the shared boundary $H$. By our previous remarks, $\delta(C, T_C) = 0$ and $(C, T_C)$ is of Type (VP4). Thus, the union of the ghost arcs of $\punct{T}_C$ with $\boundary_- \punct{C}$ is a spine of $\punct{C}$. In what follows, we use that and the other properties from (VP4) extensively.

\textbf{Case 1a:} $\punct{T}_C$ contains no ghost arcs.

In this case, $\boundary_- \punct{C}$ is a single sphere, corresponding to a vertex of $T$. In this case, $|H \cap T|$ is the degree of the vertex and so is either 3 or 4. If $(\punct{C}', \punct{T}'_C)$ is a trivial product compressionbody, then $\boundary_- \punct{C}'$ must also correspond to a vertex of $T$, by (LT2). Since $T$ is a genus 2 graph, that would imply that $(M,T)$ is a trivial $\theta$-curve and that $\mc{H} = H$ is a thrice-punctured bridge sphere. This is Conclusion (2). If $(\punct{C}', \punct{T}'_C)$ is not a trivial product compressionbody, by Lemma \ref{sphere4}, the fact that $T$ is a genus 2 graph, and that $\mc{H}^-$ contains no thrice-punctured spheres we see that $\mc{H} = H$ is a 4-punctured sphere, that $(M,T)$ is a 2-bouquet and that $\netextent(\mc{H}) = 1$. Since, $H$ is a 4-punctured sphere, it is trivial. Hence, Conclusion (3) holds.

\textbf{Case 1b:} $\punct{T}_C$ contains exactly one ghost arc $e$.

If $e$ is not a loop of $T$, then $H$ must be a sphere and the degree of each of the endpoints of $e$ is 3. Thus, $H$ is a 4-punctured sphere.  Observe that $T_C$ contains all the vertices of $T$, so $T'_C$ does not contain vertices. Since $\mc{H}^-$ contains no thrice-punctured sphere, by Lemma \ref{sphere4} applied to $(\punct{C}', \punct{T}'_C)$, we see that $\mc{H} = H$, $\boundary_- C' = \nil$, and $T_C$ is the union of two bridge arcs. In which case, $(M,T)$ is either trivial, a Hopf graph, a 2-bridge $\theta$-curve, or a 2-bridge handcuff curve. Observe that $T$ is not a 2-bouquet.

If $e$ is a loop of $T$, then $H$ must be a torus and $T$ is either a 2-bouquet or handcuff curve.  If $T$ is a handcuff curve, then $\punct{T}_C$ consists of $e$ and a single vertical arc. If $T$ is a 2-bouquet, $\punct{T}_C$ consists of $e$ and two vertical arcs. Thus, $|H \cap T| = 1$ or $|H\cap T| = 2$, respectively. Any torus component of $\boundary_- \punct{C}'$ must lie in $\mc{H}^-$ and thus, by (LT2), $(\punct{C}', \punct{T}'_C)$ is not a trivial product-compressionbody. Any thrice-punctured sphere component of $\boundary_- \punct{C}'$ must correspond to a vertex of $T$, as $\mc{H}^-$ contains no thrice-punctured spheres. Consequently, if $T$ is a handcuff curve, by Lemma \ref{toruscompbody}, $(M,T)$ is a Hopf graph and $\mc{H} = H$ is a once-punctured torus. Suppose that $T$ is a 2-bouquet. Then $T'_C$ does not contain any vertices of $T$, and so by Lemma \ref{toruscompbody}, $(\punct{C}', \punct{T}'_C)$ is either a $(T^2 \times I, \text{ bridge arc})$ or $C' = \punct{C}'$ is the result of removing an open 3-ball from a solid torus and $T'_C = \punct{T}'_C$ is the union of a ghost arc and two vertical arcs. 

With the first possibility, observe that $\boundary_- C'$ separates $M$ and that $T$ is contained entirely to one side. Furthermore, $\delta(C', T'_C) = 1$ (by direct calculation) and so there exists a leaf $(D, T_D)  \neq (C, T_C)$ of the dual digraph disjoint from $T$. But this leaf must have $\delta(D, T_C) = 0$ since $(C', T'_C)$ is the unique component of $(M,T) \setminus \mc{H}$ with $\delta = 1$. By our previous remarks, $T_D$ must contain a vertex of $T$. But this contradicts the fact that $T_C \cup T'_C = T$. Thus, this case cannot occur.

Suppose, therefore, that $C = C'$ is the result of removing an open ball from a solid torus. Let $F = \boundary_- C'$. Observe that $F \cpt \mc{H}^-$ is an essential 4-punctured sphere. Since $\boundary_- C = \nil$, the sphere $F$ is separating. Let $(\wihat{M}_1, \wihat{T}_1)$ and $\wihat{M}_2, \wihat{T}_2)$ be the result of surgering $(M,T)$ along $F$ with $\wihat{T}_1$ the graph containing the vertex of $T$. Let $\mc{H}_i = \mc{H} \cap \wihat{M}_i$. Notice that $\wihat{T}_1$ is a Hopf ringlet and that $\wihat{T}_2$ is a 2-bouquet. Thus (5) holds.

\textbf{Case 1c:} $\punct{T}_C$ contains two distinct ghost arcs $e_1$ and $e_2$.

If $H \cap T = \nil$, then by (VP4), $H$ is a genus 2 surface and $T = T_C$. If $\boundary_- C' \neq \nil$, then there is another leaf $(D, T_D) \neq (C, T_C)$ of the dual digraph. Since $T'_C = \nil$, by (LT2), $(C', T'_C)$ is not of Type (VP4) and so $\delta(C', T'_C) = 1$. Consequently, $\delta(D, T_D) = 0$. This implies that $T_D$ contains a vertex of $T$, contradicting the fact that $T = T_C$. Thus,  $\boundary_- C' = \nil$. In this case, $H = \mc{H}$ is a genus 2 Heegaard surface for the exterior of $T$ and so $(M,T)$ is either a trivial $\theta$-curve, a Hopf graph, or is knotted of low complexity. This implies that Conclusion (2) or (3) holds.

Since $(C, T_C)$ is of Type (VP4), if $e_1$ and $e_2$ are both loops, then $T = T_C$, and $H \cap T = \nil$, a possibility we have already considered. We suppose, therefore, that $e_2$, say, is not a loop. Thus, $H$ has genus 1 and either $e_1 \cup e_2$ is a core loop of the solid torus $C$ or $e_1$ is a loop based at a vertex $v$ of $T$ and $e_2$ is a ghost arc joining $v$ to another vertex $w$ of $T$ (and $T_C$). In the former case $T$ is a $\theta$-curve and in the latter case, $T$ is a handcuff curve. In either case, we have $|H \cap T| = 2$ and $T_C$ contains both vertices of $T$. We apply Lemma \ref{toruscompbody} to $(\punct{C}', \punct{T}')$. Since $\mc{H}^-$ contains no thrice-punctured spheres, no component of $\boundary_- \punct{C}'$ is a thrice-punctured sphere. By (LT2), $(C', T'_C)$ is not a trivial product compressionbody. If $\boundary_- C'$ is a torus, then as in the previous cases, we would find that $T$ was disconnected, a contradiction. Thus, by Lemma \ref{toruscompbody}, $\boundary_- C'$ is a 4-punctured sphere $F \cpt \mc{H}^-$ and $T'_C = \punct{T}'_C$ consists of two vertical arcs and a ghost arc. Since $\boundary_- C = \nil$, the sphere $F$ is separating. Surgering $(M,T)$ along $F$ results in two connected pairs $(\wihat{M}_1, \wihat{T}_1)$ and $(\wihat{M}_2, \wihat{T}_2)$. Choosing the notation so that the vertices of $T$ lie in $\wihat{T}_1$, we see that the pair $(\wihat{M}_1, \wihat{T}_1)$ is either a Hopfified $\theta$-curve or a Hopfified handcuff curve. The pair $(\wihat{M}_2, \wihat{T}_2)$ is a 2-bouquet. Thus (5) holds.

\textbf{Case 2:} $\netextent(M,T) = 1$ and $T$ is a knot.

As we have remarked, in this case, there are exactly two leaves $(C, T_C)$ and $(D, T_D)$ of the dual digraph and they both have $\delta = 1$. Since $\boundary_- C = \nil$, 
\[
1 = \delta(C, T_C) = \extent(\boundary_+ C) = \g(\boundary_+ C) - 1 + b 
\]
where $b$ is the number of components (necessarily all bridge arcs) of $T_C$. Thus, $H = \boundary_+ C$ is either a sphere with 4 punctures, a torus with two punctures, or an unpunctured genus 2 surface. Let $(C', T'_C) \cpt (M,T)\setminus \mc{H}$ be the other vp-compressionbody with $\boundary_+ C' = H$. Recall that $\mc{H}^-$ contains no thrice-punctured spheres and $T$ has no vertices. Also, by (LT2), if $(\punct{C}', \punct{T}'_C)$ is a trivial product compressionbody, then $\boundary_- \punct{C}'$ corresponds to a vertex of $T$. 

\textbf{Case 2a:} $H$ is a sphere.

By Lemma \ref{sphere4}, $H = \mc{H}$, $M = S^3$, and $T$ is either trivial or 2-bridge.

\textbf{Case 2b:} $H$ is a twice-punctured torus. 

We apply Lemma \ref{toruscompbody} to $(C', T'_C)$. Since $T$ is a knot and $\mc{H}^-$ contains no thrice-punctured spheres, one of the following occurs:
\begin{enumerate}
\item[(a)] $(C', T'_C) = (D, T_D)$ is a solid torus with a single bridge arc;
\item[(b)] $C'$ is homeomorphic to $H \times I$ and $T'_C$ is a single bridge arc;
\item[(c)] $C'$ is the result of removing an open ball from a solid torus and $\punct{T}'_C$ is the union of a ghost arc and two vertical arcs.
\end{enumerate}

If (a) holds, then $(M,T)$ is either trivial, a (1,0)-knot, a 2-bridge knot, or a (1,1) knot and $\mc{H} = H$ is a twice-punctured torus. This is Conclusion (4). If (b) holds, then $(C', T'_C) \neq (D, T_D)$ but $\delta(C', T'_C) = 1$, a contradiction. If (c) holds, let $F = \boundary_- C'$ and observe it is a 4-punctured sphere and $F \cpt \mc{H}^-$. Since $\boundary_- C = \nil$, $F$ separates $M$. Thus, surgering $(M,T)$ along $F$ produces two connected pairs $(\wihat{M}_1, \wihat{T}_1)$ and $(\wihat{M}_2, \wihat{T}_2)$. Since $T$ is a knot, both pairs are 2-bouquets. We may choose the notation so that $H \subset \wihat{M}_1$. Observe that $(\wihat{M}_1, \wihat{T}_1)$ is either a trivial 2-bouquet or a (1,1)-curve. In fact, since $F$ is c-essential in $(M,T)$, it cannot be a trivial 2-bouquet. Thus, (5) holds. 

\textbf{Case 2c:} $H$ is an unpunctured genus 2 surface.

Since $\delta(C, T_C) = 1$, $T_C = \nil$. We apply Lemma \ref{genus2disjt} to $(C', T'_C)$. If $\boundary_- C' = \nil$, then $(C', T'_C) = (D, T_D)$ and $H$ is a genus 2 Heegaard surface for $M\setminus T$. If $\boundary_- C'$ is a single 4-punctured sphere, as before we see that (5) holds with $(\wihat{M}_1, \wihat{T}_1)$ being a (1,1)-curve or (2,0)-curve that is a 2-bouquet. Suppose, therefore, that $\boundary_- C' = \boundary_- \punct{C}'$ is either one or two tori and that $T'_C$ is a ghost arc. Let $F$ and $F'$ be the components of $\boundary_- C'$ (possibly $F = F'$). 

Without loss of generality, we may assume that $H$ is oriented into $C$, so that $(C, T_C)$ is the unique sink of the dual digraph. The orientations on the edges of the dual digraph induce a partial order on the vp-compressionbodies of $(M,T)\setminus \mc{H}$ and we write $(E', T'_E) < (E, T_E)$ if there is a non-constant path, following the orientations of the edges of the dual digraph, in the dual digraph from $(E', T'_E)$ to $(E, T_E)$. The vp-compressionbody $(C, T_C)$ is the unique maximal element under this partial order and every $(D', T'_D) \cpt (M,T) \setminus \mc{H}$ has the property that $(D', T'_D) \leq (C, T_C)$. If $(D', T'_D) \neq (C, T_C)$, then also $(D', T'_D) \leq (C', T'_C)$. Let $(E, T_E)$ and $(E', T'_E)$ be the vp-compressionbodies that are distinct from $(C', T'_C)$ and which contain $F$ and $F'$ respectively. We have $(E, T_E), (E', T'_E) < (C', T'_C)$ and if $(D', T'_D) \cpt (M,T)\setminus \mc{H}$ is not either of $(C, T_C)$ or $(C', T'_C)$ then $(D', T'_D) \leq (E, T_E)$ or $(D', T'_D) \leq (E', T'_E)$. If $E \neq E'$, then $(E, T_E)$ and $(E', T'_E)$ are incomparable in the partial order.

Suppose that $\boundary_- E$ contains a component $F''$ which is not $F$ or $F'$. Let $(D', T'_D) \neq (E, T_E)$ be the vp-compressionbody with $F'' \cpt \boundary_- D'$. It cannot be $(C', T'_C)$. Thus, $(E, T_E) < (D', T'_D) < (E', T'_E)$. Likewise, if $\boundary_- E'$ contains a component which is not $F$ or $F'$, then $(E', T'_E) < (E, T_E)$. Consequently, either $(E, T_E)$ or $(E', T'_E)$ has its entire negative boundary contained in $F \cup F'$. Without loss of generality, suppose it is $(E, T_E)$. As $(E, T_E)$ is of Type (VP4), by (LT2), any component of $T_E$ incident to $F \cup F'$ must be a ghost arc $\psi$. Since $|(F \cup F') \cap T| = 2$ and $\boundary_- E \subset (F \cup F')$, we must have $\boundary_- E = F \cup F'$. Consequently, $(E, T_E) = (E', T'_E)$. Furthermore, the endpoints of $\psi$ are precisely the punctures of $F \cup F'$. Thus, $\psi \cup T'_C = T$. We see then that $\boundary_+ E$ is an unpunctured genus 2 surface. Since, apart from $(C, T_C)$ and $(D, T_D)$, every vp-compressionbody of $(M,T)\setminus \mc{H}$ is of Type (VP4), $\boundary_+ E = \boundary_+ D$ and so $(M,T)$ is a propeller knot and $\mc{H}$ is the standard propeller surface. This concludes the proof of Case 2.

It remains to show that (5) implies that $(M,T)$ is an essential Hopf slinky. We apply (5) inductively. Let $F_1 \cpt \mc{H}^-$ be the 4-punctured sphere given by (5). Surgering along $F_1$ produces two connected pairs $(M_1, T_1) = (\wihat{M}_1, \wihat{T}_1)$ and $(\wihat{M}_2, \wihat{T}_2)$ with $(M_1, T_1)$ either a Hopfified $\theta$-curve, Hopfified handcuff curve, Hopf ringlet, or 2-bouquet that is a (1,1)-curve or (2,0)-curve. The surface $H = \mc{H} \cap \wihat{M}_1$ is a twice-punctured torus in $\H(M_1, T_1)$. By Lemma \ref{crushing calc}, $\mc{H}_2 = \mc{H} \cap \wihat{M}_2$ are multiple vp-bridge surfaces continuing to satisfy (LT1), (LT2), (wLT3), (LT4), (LT5) and 
\[\begin{array}{rcl}
\netchi(\mc{H}_2) &=& \netchi(\mc{H}) - 2, \text{ and} \\
\netextent(\mc{H}_2) &=&  1\\
\end{array}
\]
Apply our previous work to the pair $(\wihat{M}_2, \wihat{T}_2)$ and the surface $\mc{H}_2$. As $F_1$ is essential, we know that $(\wihat{M}_2, \wihat{T}_2)$ is not trivial. As $\wihat{T}_2$ is a 2-bouquet, we conclude that either $(\wihat{M}_2, \wihat{T}_2)$ is a $(1,1)$-curve or $(2,0)$-curve or (5) holds for it. If it is a (1,1)-curve or (2,0) curve, then $(M,T)$ is an essential Hopf slinky of length 2. If (5) holds, we decompose it along another 4-punctured sphere $F_2$ into $(M_2, T_2)$ and $(\wihat{M}_2, \wihat{T}_2)$ such that $\mc{H}_2 \cap M_2$ is a twice-punctured torus and $(M_2, T_2)$ is a Hopf slinky (since $T_2$ was a 2-bouquet). Continuing on in this manner, we deduce that $(M,T)$ is an essential Hopf slinky and that $\mc{H}$ is the standard slinky surface.
\end{proof}

\section{Lower bounds on tunnel number and the bridge number for composite genus 2 graphs}\label{sec:lower bound}

We can now prove our lower bounds on bridge number and tunnel number. We begin with a very general result. As we previously discussed results for composite knots in \cite{TT2}, we focus on genus 2 graphs here. (Although, we could extract slightly more information even for knots). We begin by establishing notation that will be useful in the remainder of the paper. Let $(M,T)$ be an irreducible, connected pair such that $T$ is a genus 2 graph and every sphere in $M$ is separating. Suppose also that we have a prime factorization with factors $(\wihat{M}_1, \wihat{T}_1), \hdots, (\wihat{M}_n, \wihat{T}_n)$. (Recall that by Theorem \ref{unique roots} these factors are independent of the particular prime factorization.) Additionally, suppose that for each $i \in \{1, \hdots, n\}$, we have a realizable $x_i \in\Z$. Set $NE_g(i)$ and $NE_k(i)$ to be the number of factors for which $x_i = i$ and for which $\wihat{T}_i$ is a genus 2 graph or knot (respectively).

\begin{theorem}\label{fundamental theorem}
Suppose that $(M,T)$ is an irreducible composite pair such that every sphere in $M$ separates and $T$ is a genus 2 graph. Let $x$ be realizable for $(M,T)$. Then there exists a prime decomposition of $(M,T)$, such that for each of the $n$ factors $(\wihat{M}_i, \wihat{T}_i)$, there exists an admissible $x_i \geq -2$  such that
\[
x_1 + \cdots + x_n \leq x - 2(n-1)
\]
and
\[
\netextent_x(M,T) = \frac{1}{2} + \sum\limits_{i \geq 0} \big((i - \frac{1}{2}) NE_g(i)+ iNE_k(i)\big).
\]
\end{theorem}

\begin{remark}\label{rem: fundamental theorem}
By Lemma \ref{Main Ineq}, each factor $(\wihat{M}_i, \wihat{T}_i)$ contributes a non-negative integer or half integer to the sum in Theorem \ref{fundamental theorem}. By Theorem \ref{low ne class}, it contributes zero if and only if it is either a $(1,0)$-curve, a trivial $\theta$-curve, or a Hopf graph. Also, by Lemma \ref{knotlem}, a knot factor always contributes an integer.
\end{remark}

\begin{proof}
The Additivity Theorem (Theorem \ref{Add Thm}) gives a prime decomposition of $(M,T)$ into $(\wihat{M}_i, \wihat{T}_i)$ for $i = 1, \hdots n$ and integers $x_i$ such that each $x_i$ is realizable for $(\wihat{M}_i, \wihat{T}_i)$,
\[
x_1 + \cdots + x_n \leq x - 2(n-1),
\] 
and
\[
\netextent(M,T) = -p_3/2 + \sum\limits_{i=1}^n \netextent_{x_i}(\wihat{M}_i, \wihat{T}_i)
\]
where $p_3$ is the number of thrice-punctured spheres in the decomposition. 

Each factor of the decomposition is either a genus 2 graph or a knot. Let $m$ be the number of factors that are genus 2 graphs. Since $T$ is a genus 2 graph, $m \geq 1$. If there is a 2-bouquet or trivial $\theta$-graph in the decomposition, then $p_3 = 0$ and all the other factors are knots. If $p_3 \neq 0$, then it must be equal to $m-1$ since $\theta$-curves and handcuff curves each have precisely two vertices, each of degree 3. Thus, in all cases, $p_3 = m-1$. Let $I_k$ be the set of indices $i$ such that $\wihat{T}_i$ is a knot and let $I_g$ be the set of indices $i$ such that $\wihat{T}_i$ is a genus 2 graph.  Consequently,

\[
\netextent(M,T) =\frac{1}{2} + \sum\limits_{i \in I_g}\big(\netextent_{x_i}(\wihat{M}_i, \wihat{T}_i) - \frac{1}{2}\big) + \sum\limits_{i \in I_k}\netextent_{x_i}(\wihat{M}_i, \wihat{T}_i).
\]
Stratifying by the values of net extent, we have
\[
\netextent(M,T) \geq  \frac{1}{2} + \sum\limits_{i \geq 0} NE_g(i)(i - \frac{1}{2}) + NE_k(i)i.
\]
\end{proof}

We can now prove our first result on tunnel number. 

\begin{theorem}\label{tunnel result 1}
Suppose that $(M,T)$ is an irreducible composite pair such that every sphere in $M$ separates and $T$ is a genus 2 graph. Then
\[
\t(M,T) \geq \frac{m - 1}{2} + k
\]
where $m$ is the number of factors in a prime factorization that are genus 2 graphs which are not the trivial $\theta$-curves or Hopf graphs and $k$ is the number of factors that are knots which are not $(1,0)$-curves.
\end{theorem}
\begin{proof}
By the definition of tunnel number, there exists a connected $H \in \H(M,T)$ such that the genus of $H$ is $\t(M,T) + 2$ and $H$ is disjoint from $T$. Set $x = 2\t(M,T) + 2$. Observe that $\netextent(H) = x/2$. Thus, by Theorem \ref{fundamental theorem}, there exists a prime factorization of $(M,T)$ such that 
\[
\t(M,T) + 1 \geq \netextent_x(M,T) \geq \frac{1}{2} + \sum\limits_{i \geq 0} \big((i - \frac{1}{2}) NE_g(i)+ iNE_k(i)\big).
\]

By Remark \ref{rem: fundamental theorem}, 
\[
\t(M,T) \geq - \frac{1}{2} + \frac{m}{2} + k, 
\]
as desired.
\end{proof}

Similarly, for bridge number we have:
\begin{theorem}\label{bridge result 1}
Suppose that $(S^3,T)$ is an irreducible composite pair and that $T$ is a genus 2 graph. Then
\[
\b(T) \geq \frac{m+3}{2} + k 
\]
where $m$ is the number of factors that are genus 2 graphs which are not the trivial $\theta$-curve and $k$ is the number of factors that are knots. Furthermore, if equality holds then every factor in a prime factorization is a $(0,2)$-curve, trivial $\theta$-curve or trivial 2-bouquet.
\end{theorem}
\begin{proof}
The proof is nearly identical to that of Theorem \ref{tunnel result 1}, except that we start with a minimal bridge sphere $H$ for $(M,T)$. Set $x = -2 = -\chi(H)$ and observe that
\[
\b(M,T) - 1= \netextent(H) \geq \netextent_{-2}(M,T).
\]
By Theorem \ref{fundamental theorem}, for each factor $(\wihat{M}_i, \wihat{T}_i)$, there exists an even integer $x_i \geq -2$ such that 
\begin{equation}\label{xsum}
x_1 + \cdots + x_n \leq x - 2(n-1) = -2n
\end{equation}
and
\[
\b(M,T) - 1 \geq  \frac{1}{2} + \sum\limits_{i \geq 0} \big((i - \frac{1}{2}) NE_g(i)+ iNE_k(i)\big) \geq \frac{1}{2} + \frac{m'}{2} + k',
\]
The number $m'$ is the number of genus 2 graph factors $(\wihat{M}_i, \wihat{T}_i)$ for which $\netextent_{x_i}(\wihat{M}_i, \wihat{T}_i) > 1/2$ and $k'$ is the number of knot factors $(\wihat{M}_i, \wihat{T}_i)$ for which $\netextent_{x_i}(\wihat{M}_i, \wihat{T}_i) > 0$. Since $M = S^3$, no knot factor is a (1,0)-curve and so, by Theorem \ref{low ne class}, $k = k'$. We would like to improve the inequality by showing that $m' = m$. 

Since each $x_i \geq -2$ by Inequality \ref{xsum}, each $x_i = -2$. Suppose we have a factor $(\wihat{M}_i, \wihat{T}_i)$ for which $\netextent_{x_i}(\wihat{M}_i, \wihat{T}_i) = 1/2$. Let $\mc{H}_i$ be a locally thin multiple vp-bridge surface for $(\wihat{M}_i, \wihat{T}_i)$ with $\netchi(\mc{H}_i) = -2$ and $\netextent(\mc{H}_i) = 1/2$.  By Theorem \ref{low ne class}, $(\wihat{M}_i, \wihat{T}_i)$ is either a trivial $\theta$-curve or a Hopf graph and $\mc{H}_i$ is a 3 or 4 times punctured sphere. Suppose $\wihat{T}_i$ is a Hopf graph. If the sphere $\mc{H}_i$ separates the vertices of $\wihat{T}_i$, then each loop of $\wihat{T}_i$ intersects $\mc{H}_i$ at least twice and the separating edge intersects it at least once, a contradiction. If $\mc{H}_i$ does not separate the vertices, then each loop of $\wihat{T}_i$ intersects $\mc{H}_i$ twice, and so $\mc{H}_i$ is a four-punctured sphere. But in this case, $\netextent(\mc{H}_i) = 1$, a contradiction. Thus, $\wihat{T}_i$ is not a Hopf graph and so $m' = m$.

Observe that if equality holds, then $NE_g(1) + NE_k(1) = m + k$. Since each $x_i = -2$ the result follows from Theorem \ref{low ne class}.
\end{proof}

For Brunnian $\theta$-curves a more careful analysis gives stronger results. 

\begin{theorem}\label{Brunnian result}
Suppose that $T \subset S^3$ is a composite Brunnian $\theta$-curve with $m$ factors in its prime decomposition. Then
\[
\t(S^3,T) \geq m
\]
and
\[
\b(T) \geq m + \frac{3}{2}
\]
\end{theorem}

\begin{proof}
We begin by showing that each factor of a Brunnian $\theta$-curve is also Brunnian. Suppose that $T \subset S^3 = M$ is a Brunnian $\theta$-curve. Since it is a $\theta$-curve and every $S^2$ separates $S^3$, the pair $(S^3, T)$ is irreducible and has no (lens space, core) summands. Suppose that $F \subset (M,T)$ is an essential sphere. If $F$ does not separate the vertices of $T$, each edge of $T$ must intersect $F$ an even number of times. If it does separate the vertices of $T$, then each edge intersects $F$ an odd number of times. If $F$ is twice punctured, it intersects a single edge $e_3$ of $T$. The other two edges and both vertices then lie on the same side of $F$. Let $e_1$ be one of the other edges of $T$.  Since $T$ is Brunnian, the cycle $e_1 \cup e_3$ is the unknot $\tau$ and $F$ gives a connected sum decomposition of $\tau$. Thus, both components of $(e_1 \cup e_2) \setminus F$ must be arcs parallel into $F$. In particular, this means that $F$ is $\boundary$-parallel to the side not containing $e_1$. Thus, $(M,T)$ contains no essential twice-punctured spheres. If $F$ is a thrice-punctured sphere, each edge of $T$ intersects $F$ exactly once and again, $F$ is a connected summing sphere on the cycles of $T$. Thus, surgering $(M,T)$ along $F$ produces two pairs, each a Brunnian $\theta$-graph in $S^3$. In particular, every factor of the given $(M,T)$ is a Brunnian $\theta$-curve. Let $m$ be the number of factors.

To prove the statement for tunnel number, set $x = 2\t(M,T)  + 2$, this is the negative Euler characteristic of a minimal Heegaard surface for the exterior of $T$ and set $e = x/2$; this is its net extent. To prove the statement for bridge number, set $x = -2$ (the negative Euler characteristic of a sphere) and $e = (x + 2\b(M,T))/2$; this is its net extent. As in the proofs of Theorem \ref{fundamental theorem} and \ref{bridge result 1},  for each factor $(\wihat{M}_i, \wihat{T}_i)$ (necessarily a prime, Brunnian $\theta$-curve) of $(M_i, T_i)$ there exists an even integer $x_i \geq -2$ such that 
\[
e \geq \frac{1}{2} + \sum\limits_{i \geq 0} \big((i - \frac{1}{2}) NE(i)\big).
\]
As $T$ is Brunnian, there is no knot factor or trivial $\theta$-graph factor. Consequently, $NE(0) = NE(1/2) = 0$. Suppose that some factor $(\wihat{M}_i, \wihat{T}_i)$ has $\netextent_{x_i}(\wihat{M}_i, \wihat{T}_i) = 1$. Recalling that $\wihat{T}_i$ is a $\theta$-graph, we see that by Theorem \ref{low ne class}, the pair is either knotted of low complexity or an essential Hopf slinky. By Corollary \ref{smallBrun} it cannot be knotted of low complexity and by Proposition \ref{no Brunnian slinky}, it cannot be an essential Hopf slinky. Consequently, $NE(1) = 0$. Thus,
\[
e  \geq \frac{1}{2} + \big(\frac{3}{2} - \frac{1}{2}\big)m = \frac{1}{2} + m.
\]

In the tunnel number case, this produces
\[
\t(S^3,T) \geq m - \frac{1}{2}.
\]
Since both tunnel number and $m$ are integers, $\t(M,T) \geq m$ as desired.

In the bridge number case, this produces
\[
\b(T) \geq \frac{3}{2} + m.
\]
Unlike tunnel number, the bridge number need not be an integer.
\end{proof}

Using a different analysis, we can prove Morimoto's bound for m-small pairs. This is very similar to \cite[Theorem 7.3]{TT2}.

\begin{theorem}\label{thm: Morimoto}
Suppose that $(M,T)$ is an irreducible, composite pair where every sphere in $M$ is separating and $T$ a $\theta$-curve or handcuff curve. Let $(\wihat{M}_1, \wihat{T}_1), \cdots, (\wihat{M}_n, \wihat{T}_n)$ be the factors of a prime factorization of $(M,T)$ and suppose that each is m-small. Then
\[
\t(M,T) \geq \t(\wihat{M}_1, \wihat{T}_1) + \cdots + \t(\wihat{M}_n, \wihat{T}_n).
\]
\end{theorem}
\begin{proof}
Let $H \in \H(M,T)$ be a minimal genus Heegaard surface for $M\setminus T$. Set $x = 2\t(M,T) + 2$ and recall that $\netextent_x(H) = x/2$. By Theorems \ref{Add Thm} and \ref{unique roots}, for each $i$, there exists a realizable integer $x_i$ for $(\wihat{M}_i, \wihat{T}_i)$ so that
\[
\t(M,T) + 1 \geq \netextent_x(M,T) \geq - \frac{p_3}{2} + \sum\limits_{i=1}^n \netextent_{x_i}(\wihat{M}_i, \wihat{T}_i)
\]
where $p_3$ is the number of trivalent vertex sums in the decomposition. 

For each $i$, choose a locally thin $H_i \in \H(\wihat{M}_i,\wihat{T}_i)$ such that $\netchi(H_i) \leq x_i$ and $\netextent(H_i)  = \netextent_{x_i}(\wihat{M}_i, \wihat{T}_i)$. Without loss of generality, we may assume that $x_i = \netchi(H_i)$. Since each $(\wihat{M}_i, \wihat{T}_i)$ is m-small, $\mc{H}_i^{-} = \nil$. Thus, $H_i$ is connected. Let $p = |H_i \cap \wihat{T}_i|$. If $H_i$ does not separate the vertices of $\wihat{T}_i$ (or if $\wihat{T}_i$ is a knot), then $p$ is even and there are $b = p/2$ bridge arcs of $\wihat{T}_i \setminus H_i$ on one side of $H_i$. If $H_i$ does separate the vertices of $\wihat{T}_i$ and $\wihat{T}_i$ is a handcuff curve, then the loops of $\wihat{T}_i$ each intersect $H_i$ an even number of times, while the separating edge intersects $H_i$ an odd number of times. In that case, there are $b = (p-1)/2$ bridge arcs on each of the two sides of $H_i$. If $H_i$ separates the vertices of $\wihat{T}_i$ and $\wihat{T}_i$ is a $\theta$-curve, each edge of $\wihat{T}_i$ intersects $H_i$ an odd number of times. In this case, there are $b = (p-3)/2$ bridge arcs on either side of $H_i$. Let $\epsilon_i = 1 - \chi(\wihat{T}_i)$.

In each case, successively tube along bridge arcs, all on the same side of $H_i$, to create a connected surface $H'_i \in \H(\wihat{M}_i, \wihat{T}_i)$ of genus
\[
g'_i = (x_i + 2)/2 + b
\]
Observe that $H'_i$ separates the vertices of $\wihat{T}_i$ if and only if $H_i$ does. If $H'_i$ does not separate the vertices of $\wihat{T}_i$, we see that $H'_i$ is a Heegaard surface for $\wihat{M}_i \setminus \wihat{T}_i$. In which case, observe
\[
\t(\wihat{M}_i, \wihat{T}_i) \leq  g'_i - \epsilon_i = \netextent_{x_i}(\wihat{M}_i, \wihat{T}_i)+\chi(\wihat{T}_i).
\]

If $H'_i$ separates the vertices of $\wihat{T}_i$ and $\wihat{T}_i$ is a handcuff curve, then on either side of $H'_i$, the graph $\wihat{T}_i$ consists of a single vertical arc and a ghost arc that is a loop based at a single vertex. Attach the frontier of a neighborhood of one of these vertical arcs and ghost arcs to $H'_i$ to create $H''_i$. Notice that $H''_i$ is a Heegaard surface for $\wihat{M}_i \setminus \wihat{T}_i$. It has genus equal to $g'_i + 1$. We have, therefore,
\[
\t(\wihat{M}_i, \wihat{T}_i) \leq  (g'_i + 1) - \epsilon_i  = \netextent_{x_i}(\wihat{M}_i, \wihat{T}_i) - \frac{1}{2}
\]

If $H'_i$ separates the vertices of $\wihat{T}_i$ and $\wihat{T}_i$ is a $\theta$-curve, then on each side of $H'_i$, the graph $\wihat{T}_i$ consists of a single vertex and three vertical arcs. Choose a side and attach to $H'_i$ the frontier of a neighborhood of the arcs on one side to create a Heegaard surface $H''_i$ for $\wihat{M}_i \setminus \wihat{T}_i$. It has genus $g'_i + 2$. We have, therefore, 
\[
\t(\wihat{M}_i, \wihat{T}_i) \leq (g'_i + 1) - \epsilon_i =  \netextent_{x_i}(\wihat{M}_i, \wihat{T}_i) - \frac{1}{2}
\]
Thus, 
\[
\t(M,T) +  1 \geq -\frac{p_3}{2} + \sum\limits_{i} \left(\t(\wihat{M}_i, \wihat{T}_i)  + \frac{1}{2}\right) + \sum\limits_{j} \big(\t(\wihat{M}_j, \wihat{T}_j)\big) 
\]
where the first sum is over all $i$ such that $\wihat{T}_i$ is a genus 2 graph and the second over all $j$ such that $\wihat{T}_j$ is a knot. Observe that $p_3$ is one less than the number of factors that are genus 2 graphs. 
Hence, letting $m$ be the number of genus 2 graph factors,
\[
\t(M,T)  \geq -\frac{1}{2} + \sum\limits_{i=1}^n \t(\wihat{M}_i, \wihat{T}_i).
\]
Since tunnel number is an integer,
\[
\t(M,T)  \geq \sum\limits_{i=1}^n \t(\wihat{M}_i, \wihat{T}_i).
\]
\end{proof}

\section{Achieving Equality}\label{sec:equality}

In this section, we study the situation when $\t(M,T)$ achieves the lower bound in Theorem \ref{tunnel result 1}. Recall from \cite{MT} that any two prime decompositions of an irreducible pair $(M,T)$ with $T$ a genus 2 graph and $M$ a compact 3-manifold without nonseparating 2-spheres have the same set of factors. Let $N_k(g,b)$ and $N_g(g,b)$ be the number of factors that are $(g,b)$-curves that are knots or graphs, respectively. Let $N(\text{Hopf})$ be the number of Hopf graph factors  and $N(\pi)$ be the number of factors that are propeller knots which are not also essential Hopf slinkies. For an essential Hopf slinky $\sigma$, let $\ell(\sigma)$ denote its length. Also note that the quantity $-\chi(\sigma)$ is 1 if $\sigma$ is a genus 2 graph and 0 if it is a knot. Let $N(\text{tr2bq})$ be the number of trivial 2-bouquets in the factorization. (This is either 0 or 1.)

\begin{theorem}\label{odd bound}
Let $(M,T)$ be a connected, irreducible pair such that  every sphere in $M$ separates, $T$ is a genus 2 graph and  $(M,T)$ is composite. Suppose that a (and hence every) prime factorization of $(M,T)$ has $n$ factors of which $m$ are genus 2 graphs which are not trivial $\theta$-graphs or Hopf graphs and $k$ are knots which are not $(1,0)$ knots. If
\[
\t(M,T) = \frac{m-1}{2} + k,
\]
then every factor is either a trivial $\theta$-curve, Hopf graph, trivial 2-bouquet, (0,2)-curve,  (1,0)-curve, (1,1)-curve, (2,0)-curve, Hopf slinky, or propeller knot. Furthermore, the following holds:
\[\begin{array}{ll}
N_g(1,1) + 2N(1,0) + 2N(\text{Hopf}) + 3N_g(2,0) &\\
 + 2N_k(2,0) + 4N(\pi) + \sum\limits_{\sigma}(\ell(\sigma) - \chi(\sigma)) &\leq\\
  3 + N(\text{tr2bq}) + N_g(0,2) + 2N_k(0,2)&\\
  \end{array}
\]
where the sum is over all Hopf slinkies $\sigma$.
\end{theorem}
\begin{proof}
We continue the argument of Theorems \ref{fundamental theorem} and \ref{tunnel result 1}, adapting them slightly. Since $t = \t(M,T) = (m-1)/2 + k$, there exists a Heegaard surface $H$ for $M \setminus T$ with $-\chi(H) = 2t + 2$ and $\extent(H) = t+1$. As in the earlier theorems, for each $i$, there exists an even integer $x_i \geq -2$ such that $x_i$ is realizable for $(\wihat{M}_i, \wihat{T}_i)$ and
\[
x_1 + \cdots + x_n \leq x - 2(n-1) = m + 2k - 2n + 3
\]

We also have
\[
\t(M,T) \geq -1 + \netextent_x(M,T) = -\frac{1}{2} + \sum\limits_{i \geq 0} \big((i - 1/2)NE_g(i) + iNE_k(i)\big).
\]
As in Theorem \ref{tunnel result 1}, 
\[
\t(M,T) \geq -\frac{1}{2} + \sum\limits_{i \geq 0} \big((i - 1/2)NE_g(i) + iNE_k(i)\big) \geq \frac{m-1}{2} + k.
\]
Note that equality holds only if every factor $(\wihat{M}_i, \wihat{T}_i)$ has $\netextent_{x_i}(\wihat{M}_i, \wihat{T}_i) \leq 1$. Thus, by Theorem \ref{low ne class}, each factor is a trivial $\theta$-curve, Hopf graph, trivial 2-bouquet,  knotted of low complexity, propeller knot, or essential Hopf slinky. Furthermore, suppose $(\wihat{M}_i, \wihat{T}_i)$ is a factor. Then $\netextent_{x_i}(\wihat{M}_i , \wihat{T}_i) = 0$ if and only if it is a $(1,0)$-curve. It is a trivial $\theta$-curve or Hopf graph if and only if $\netextent_{x_i}(\wihat{M}_i, \wihat{T}_i) = 1/2$. It follows that if $\wihat{T}_i$ is a trivial $\theta$-curve then $x_i = -2$ and if $\wihat{T}_i$ is a (1,0)-curve or Hopf graph then $x_i = 0$. 

For each factor $(\wihat{M}_i, \wihat{T}_i)$, let $\mc{H}_i$ be a locally thin multiple vp-bridge surface such that $x'_i = \netchi(\mc{H}_i) \leq x_i$ and $\netextent(\mc{H}_i) = \netextent_{x_i}(\wihat{M}_i, \wihat{T}_i)$. Note that 
\begin{equation}\label{x'ineq}
x'_1 + \cdots + x'_n \leq m + 2k - 2n +3.
\end{equation}

By Theorem \ref{low ne class} and the definition of each type of spatial graph, the following hold for each factor:
\begin{itemize}
\item if $(\wihat{M}_i, \wihat{T}_i)$ is a trivial $\theta$-curve, trivial 2-bouquet, or $(0,2)$-curve then $\mc{H}_i$ is a sphere;
\item if $(\wihat{M}_i, \wihat{T}_i)$ is a (1,0)-curve, Hopf graph, or (1,1)-curve, then $\mc{H}_i$ is a torus;
\item if $(\wihat{M}_i, \wihat{T}_i)$ is a (2,0)-curve then $\mc{H}_i$ is an unpunctured genus 2 surface;
\item if $(\wihat{M}_i, \wihat{T}_i)$ is an essential Hopf slinky, then $\mc{H}_i$ is the standard slinky surface.
\item if $(\wihat{M}_i, \wihat{T}_i)$ is a propeller knot that is not an essential Hopf slinky, then $\mc{H}_i$ is the standard propeller surface.
\end{itemize}
Correspondingly, we conclude that
\begin{itemize}
\item $x'_i = -2$ if and only if $(\wihat{M}_i, \wihat{T}_i)$ is a trivial $\theta$-curve, trivial 2-bouquet, or $(0,2)$-curve;
\item $x'_i = 0$ if and only if $(\wihat{M}_i, \wihat{T}_i)$ is a (1,0)-curve, Hopf graph, or (1,1)-curve;
\item $x_i = 2$ if and only if $(\wihat{M}_i, \wihat{T}_i)$ is a $(2,0)$-curve or essential Hopf slinky of length 2;
\item $x_i = 4$ if and only if $(\wihat{M}_i, \wihat{T}_i)$ is an essential Hopf slinky of length 4 or a propeller knot that is not an essential Hopf slinky;
\item $x_i \geq 6$ if and only if $(\wihat{M}_i, \wihat{T}_i)$ is an essential Hopf slinky of length $x_i \geq 6$.
\end{itemize}

Let $N(tr\theta)$ be 1 if the factorization contains a trivial $\theta$ curve and 0 otherwise. Let $N_g(\sigma)$ be the number of essential Hopf slinkies that are genus 2 graphs,

Thus,
\[\begin{array}{rl}
\sum x'_i =& -2(N(tr\theta) + N(\text{tr2bq})) -2N_g(0,2) - 2N_k(0,2) \\
&+ 2N_g(2,0) + 2N_k(2,0) + 4N(\pi) + \sum\limits_{\sigma}\ell(\sigma),\\
\end{array}
\]
and
\[\begin{array}{rl}
3 +m + 2k - 2n &=\\
3 + 2(m+ k - n) - m &=\\
3 -2\big(N(0,1) + N(tr\theta) + N(\text{Hopf})\big) &\\
-\big(N_g(0,2) + N_g(1,1) + N_g(2,0) + N_g(\sigma) + N(\text{tr2bq})\big).
 \end{array}
\]
Plugging into Inequality \eqref{x'ineq} and rearranging, we obtain the desired inequality.
\end{proof}

\begin{corollary}\label{thetafact}
Suppose that $(M,T)$ is a connected, irreducible, composite pair with $T$ a $\theta$-graph such that no factor of $(M,T)$ is a knot or $(0,2)$-curve. If $(M,T)$ has $m$ factors and $\t(M,T)= \frac{m-1}{2}$, then $T$ has exactly 3 factors and they are all $(1,1)$-curves.
\end{corollary}
\begin{proof}
Notice that $m$ must be odd. Since no factor of $(M,T)$ is a knot, no factor is a trivial $\theta$-curve. The result follows from Theorem \ref{odd bound}, after observing that $k = N(2,0) = 0$, that $\ell(\sigma) \geq 2$ for any Hopf slinky, and the hypothesis that there are at least two factors in a prime decomposition of $(M,T)$. 
\end{proof}

We conclude by analyzing the distribution of knotted curves of low complexity when tunnel number is minimized relative to the number of factors. For convenience, we restrict to the case when $T$ is a $\theta$-curve or handcuff curve. With some slight modifications we could also deduce a version for knots or 2-bouquets. 

\begin{corollary}\label{cor: distribution}
Suppose that $(M,T)$ is a composite, connected, irreducible pair such that every sphere in $M$ is separating and $T$ is a genus 2 graph.  Suppose that $(M,T)$ has $n$ factors, of which $m$ are genus 2 graphs that are not the trivial $\theta$-curve or a Hopf graph and $k$ of which are knots that are not $(1,0)$-curves.  If
\[
\t(M,T) = \frac{m-1}{2} + k
\]
then all factors have net extent at most 1 and, 
\begin{enumerate}
\item the number of factors that are trivial $\theta$-curves, trivial 2-bouquets, Hopf graphs,  knotted of low complexity, or Hopf slinkies of length 2 is at least $(4n-3)/6$.
\item the number of factors that are trivial $\theta$-curves, trivial 2-bouquets, $(0,2)$-curves, $(1,0)$-curves, Hopf graphs, or $(1,1)$-curves is at least $(2n-3)/4$.
\item the number of factors that are trivial $\theta$-curves, trivial 2-bouquets, $(0,2)$-curves, or (1,1)-knots is at least $(n-3)/3$. 
\end{enumerate}
\end{corollary}
\begin{proof}
Let $n_-$ be the number of factors that are trivial 2-bouquets or (0,2)-curves. Let $n_0$ be the number that are Hopf graphs, (1,0)-curves, or (1,1)-curves that are genus 2 graphs. Let $n_2$ be the number that are (2,0)-curves or Hopf slinkies of length  2 and let $n_+$ be the number that are propeller knots or Hopf slinkies of length greater than 2. Recalling that the length of a Hopf slinky is an even integer which is at least 2, the inequality in the Conclusion of Theorem \ref{odd bound} implies
\[
(*) \hspace{.5in} n_0 + 2n_2 + 4n_+ \leq 3 + 2n_-
\]
 Observe that
\[
n = n_- + n_0 + n_2 + n_+ + N_k(1,1) + N(tr\theta)
\]
where $N(tr\theta)$ is 0 if no factor is a trivial $\theta$-curve and 1 if there is such a factor. Thus, from Inequality (*), we obtain the following inequalities:
\[
\begin{array}{rcl}
n + n_2 + 3n_+ &\leq& 3 + 3n_- + N_k(1,1) + N(tr\theta) \\
2n + 2n_+ & \leq &3 + 4n_- + n_0 + 2N_k(1,1) + 2N(tr\theta)\\
4n &\leq& 3 + 6n_- + 3n_0 + 2n_2 + 4N_k(1,1) + 4N(tr\theta).\\
\end{array}
\]
Move the constant 3 to the left, decrease the left hand side of each of those inequalities and increase the right hand side to obtain:
\[
\begin{array}{rcl}
n -3&\leq& 3(n_- + N_k(1,1) + N(tr\theta))\\
2n - 3 & \leq &  4(n_- + n_0 + N_k(1,1) + N(tr\theta))\\
4n - 3 & \leq & 6(n_- + n_0 + n_2 + N_k(1,1) + N(tr\theta)).\\
\end{array}
\]
These imply the inequalities we were looking for.
\end{proof}

\section*{Acknowledgements}
Taylor was supported by a Colby College Research Grant. Tomova was supported by an NSF research grant.

\end{document}